\documentclass[11pt,english]{article}
\usepackage{amsfonts,amsmath,amsthm,amscd,amssymb,latexsym}

\usepackage[T2A]{fontenc}
\usepackage[cp1251]{inputenc}
\usepackage{graphicx}

\begin{document}

\title{Generalized Cauchy-Riemann equations\\ and relevant PDE}

\author{V. Gutlyanskii, V. Ryazanov, A. Salimov, R. Salimov}

\theoremstyle{plain}
\newtheorem{theorem}{Theorem}[section]
\newtheorem{lemma}{Lemma}[section]
\newtheorem{proposition}{Proposition}[section]
\newtheorem{corollary}{Corollary}[section]
\theoremstyle{definition}

\newtheorem{example}{Example}[section]
\newtheorem{remark}{Remark}[section]
\newcommand{\keywords}{\textbf{Key words.  }\medskip}
\newcommand{\subjclass}{\textbf{MSC 2000. }\medskip}
\renewcommand{\abstract}{\textbf{Abstract.  }\medskip}
\numberwithin{equation}{section}

\setcounter{section}{0}
\renewcommand{\thesection}{\arabic{section}}
\newcounter{unDef}[section]
\def\theunDef{\thesection.\arabic{unDef}}
\newenvironment{definition}{\refstepcounter{unDef}\trivlist
\item[\hskip \labelsep{\bf Определение \theunDef.}]}%
{\endtrivlist}

\maketitle

\def\Xint#1{\mathchoice
   {\XXint\displaystyle\textstyle{#1}}
   {\XXint\textstyle\scriptstyle{#1}}
   {\XXint\scriptstyle\scriptscriptstyle{#1}}
   {\XXint\scriptscriptstyle\scriptscriptstyle{#1}}
   \!\int}
\def\XXint#1#2#3{{\setbox0=\hbox{$#1{#2#3}{\int}$}
     \vcenter{\hbox{$#2#3$}}\kern-.5\wd0}}
\def\dashint{\Xint-}

\newcommand{\M}{{\cal M}}
\newcommand\disk{{\Bbb D}}

\def\esssup{\mathop{\rm {ess\,sup}}\limits}
\def\ink{\mathop{\int\int}\limits}
\def\h{{\Bbb H}}
\def\R{{\rm Re\,}}
\def\I{{\rm Im\,}}
\def\mod{{\rm mod\,}}
\def\e{\varepsilon}
\def\QED{{\hfill $\Box$\par\bigskip}}

\pagestyle{myheadings}


\begin{abstract}
Here we give a survey of consequences from the theory of the
Beltrami equations in the complex plane $\mathbb C$ to generalized
Cauchy-Riemann equations $\nabla v = B \nabla u$ in the real plane
$\mathbb R^2$ and clarify the relationships of the latter to the
$A-$harmonic equation ${\rm div} A\,{\rm grad}\, u = 0$ with matrix
valued coefficients $A$ that is one of the main equations of the
potential theory, namely, of the hydro\-mechanics (fluid mechanics)
in anisotropic and inhomogeneous media. The survey includes various
types of results as theorems on existence, representation and
regularity of their solutions, in particular, for the main boundary
value problems of Hilbert, Dirichlet, Neumann, Poincare and Riemann.
\end{abstract}

\bigskip
{\bf MSC 2020.} {Primary 30C62, 30C65, 30E25 Secondary 35J25, 76-02}

\bigskip
{\bf Key words.} Cauchy-Riemann system, boundary value problems of
Hilbert, Dirichlet, Neumann, Poincare and Riemann, Beltrami
equations, $A-$harmonic equations, generalized Cauchy-Riemann
equations.


\section{Introduction} By default, we often use here the isomorphism of the real plane
$\mathbb R^2$ and the complex plane $\mathbb C$, implicating the
natural one-to-one correspondence $Z:=(x,y)\leftrightarrow
z:=x+i\,y$. As it is well-known, the characteristic property of an
analytic function $f=u+i\,v$ in the complex plane $\mathbb C$ is
that its real and imaginary parts satisfy the {\bf Cauchy-Riemann
system}
\begin{equation}\label{CR}
\frac{\partial u}{\partial x}\ =\ \frac{\partial v}{\partial y}\ , \
\ \ \frac{\partial u}{\partial y}\ =\ -\,\frac{\partial v}{\partial
x}\ .
\end{equation}
Euler was the first who has found the connection of the system
(\ref{CR}) to the analytic functions.

\medskip

A physical interpretation of (\ref{CR}), going back to Riemann works
on function theory, is that $u$ represents a {\bf potential
function} of the incompressible fluid steady flow in the plane and
$v$ is its {\bf stream function}.

\medskip

This system can be rewritten as the one equation in the matrix form
\begin{equation}\label{M}
\nabla\, v\ =\ {\mathbb H}\ \nabla\, u\ ,
\end{equation}
where $\nabla\, v$ and $\nabla\, u$ denotes the gradient of $v$ and
$u$, correspondingly, interpreted as vector-columns in $\mathbb
R^2$, and ${\mathbb H}:\mathbb R^2\to\mathbb R^2$ is the so-called
{\bf Hodge star operator} represented as the ${2\times 2}$ matrix
\begin{equation}\label{H}
{\mathbb H}\  :=\ \left[\begin{array}{ccc} 0  & -1 \\
1 & 0 \end{array}\right]\ ,
\end{equation}
which carries out the counterclockwise rotation of vectors by the
angle ${\pi}/{2}$ in the real plane $\mathbb R^2$. Thus, (\ref{M})
shows that streamlines and equipotential lines of the fluid flow are
mutually orthogonal.

\medskip

Note that $\mathbb H$ is an analog of the imaginary unit in the
space $\mathbb {\bf M^{2\times 2}}$ of all ${2\times 2}$ matrices
with real entries because
\begin{equation}\label{U}
{\mathbb H}^2\  :=\ -\,I\ ,
\end{equation}
where I is the unit ${2\times 2}$ matrix. Moreover, recall that the
multiplication by the imaginary unit $i=e^{i\frac{\pi}{2}}$ also
means the counterclockwise rotation of vectors by the angle
${\pi}/{2}$ but in the complex plane $\mathbb C$.

\medskip

Elementary algebraic calculations show that the Cauchy-Riemann
system (\ref{CR}) in the complex form means that the function
$f:=u+i\,v$ satisfy the equation
\begin{equation}\label{C}
f_{\bar z}\ =\ 0
\end{equation}
with the formal complex derivative
\begin{equation}\label{P}
\frac{\partial}{\partial\bar z}\ :=\ \frac{1}{2}\left(
\frac{\partial}{\partial x}\ +\ i\ \frac{\partial}{\partial
y}\right)\ ,\ \ \ z\ :=\ x\ +\ i\, y\ .
\end{equation}

The relation (\ref{C}) is a necessary condition for the existence of
the usual complex derivative
\begin{equation}\label{CC}
f^{\prime}(z)\ :=\ \lim\limits_{\Delta z\to 0}\frac{f(z+\Delta
z)-f(z)}{\Delta z}
\end{equation}
and for $f$ to be an analytic function in a domain $D$ of $\mathbb
C$. Inversely, a continuous function $f:D\to\mathbb C$ with the
generalized derivative by Sobolev $f_{\bar z}=0$ a.e. is an analytic
function, see e.g. Lemma 1 in \cite{ABe}.

\medskip

Note that the equation (\ref{C}) in a domain $D$ of $\mathbb C$ is a
partial case of the {\bf Beltrami equations}
\begin{equation}\label{B}
f_{\bar z}\ =\ \mu(z)\, f_z\ ,
\end{equation}
where $\mu:D\to\mathbb C$ is a measurable function with $|\mu(z)|<1$
a.e., and
\begin{equation}\label{D}
\frac{\partial}{\partial z}\ :=\ \frac{1}{2}\left(
\frac{\partial}{\partial x}\ -\ i\ \frac{\partial}{\partial
y}\right)\ .
\end{equation}

In this paper we study the interconnections between the Beltrami
equations (\ref{B}) and the {\bf generalized Cauchy-Riemann
equations} of the form
\begin{equation}\label{BM}
\nabla\, v\ =\ B\ \nabla\, u
\end{equation}
with the matrix valued coefficients $B:D\to M^{2\times 2}$. Then, on
the basis of the well developed theory of the Beltrami equations, we
give the corresponding consequences for the equations (\ref{BM})
that describe flows in anisotropic and inhomogeneous media.

\medskip

Moreover, let us clarify the relationships of the equations
(\ref{BM}) and the {\bf $A-$harmonic equation}
\begin{equation}\label{DIV}
{\rm div} A(Z)\,{\rm grad}\, u(Z) = 0\ ,\ \ \ Z=(x,y)\,\in\,\mathbb
R^2\ ,
\end{equation}
with matrix valued coefficients $A:D\to M^{2\times 2}$ that is one
of the main equations of hydro\-mechanics in anisotropic and
inhomogeneous media.

\medskip

For this purpose, recall that the Hodge operator $\mathbb H$
transforms curl-free fields into divergence-free fields and vice
versa. Thus, if $u\in W^{1,1}_{\rm loc}$ is a solution of
(\ref{DIV}) in the sense of distributions, then the field
$V:={\mathbb H}A\,\nabla u$ is curl-free and, consequently,
$V=\nabla v$ for some $v\in W^{1,1}_{\rm loc}$ and the pair $(u,v)$
is a solution of the equation (\ref{BM}) in the sense of
distributions with
\begin{equation}\label{HA}
B\ :=\ \,{\mathbb H}\cdot A\ .
\end{equation}

\medskip

Vice versa, if $u$ and $v\in W^{1,1}_{\rm loc}$ satisfies (\ref{BM})
in the sense of distributions, then $u$ satisfies (\ref{DIV}) also
in the sense of distributions with
\begin{equation}\label{HB}
A\ :=\ -{\mathbb H}\cdot B\ =\ {\mathbb H}^{-1}\cdot B
\end{equation}
because the curl of any gradient field is zero in the sense of
distributions.

\medskip

The next section is devoted to the interconnections of the complex
coefficients $\mu$ in (\ref{B}) and the matrix valued coefficients
$B:D\to M^{2\times 2}$ in (\ref{BM}), and to the characterization of
the relevant matrices $B$ in (\ref{BM}).

\section{Preliminary algebraic calculations and remarks}

Theorem 16.1.6 in \cite{AIM} stated the interconnections between the
equations (\ref{B}) and (\ref{DIV}) and makes it is possible to
calculate the matrix coefficient $A$ in (\ref{DIV}) through the
complex coefficient $\mu$ in (\ref{B}) and vice versa.

Thus, in view of comments around the relations (\ref{HA}) and
(\ref{HB}) in Introduction, we could apply this theorem to express
the matrix coefficient $B$ in (\ref{BM}) through the complex
coefficient $\mu$ in (\ref{B}) and vice versa. However, the
mentioned comments involved the regularity of functions $f$, $u$ and
$v$ in the Sobolev class $W^{1,1}_{\rm loc}$, i.e., the existence of
the corresponding generalized derivatives.

In fact, the interconnections between the coefficients $B$ and $\mu$
are purely algebraic in nature and they can be established
pointwise. Therefore, although the corresponding computations are
long enough, we carry them out here directly, providing additional
verification of the following.

\begin{proposition}\label{pr} Let $D$ be a domain in $\mathbb C$ and a function $f:D\to\mathbb
C$ have partial derivatives in $x$ and $y$, $z=x+i\,y$, at a point
$z_0\in D$. Suppose that $|f_z(z_0)|>|f_{\bar z}(z_0)|$. Then the
functions $u:={\rm Re}\,f$ and $v:={\rm Im}\,f$ satisfy the equality
(\ref{BM}) at the point $z_0$ with the matrix coefficient
\begin{equation}\label{Bmu}
B\ =\ B^{\mu}\ :=\ \begin{pmatrix}
\frac{2 \, \rm{Im}\,  \mu}{1-|\mu|^2} & -\frac{|1+\mu|^2}{1-|\mu|^2} \\
\frac{|1-\mu|^2}{1-|\mu|^2} & -\frac{2 \, \rm{Im}\,
\mu}{1-|\mu|^2}\,
\end{pmatrix}\ ,\ \ \ \mu\ :=\ \frac{f_{\bar
z}(z_0)}{f_z(z_0)}\ .
\end{equation}
\end{proposition}

\begin{proof} Setting $\mu_1={\rm Re}\,\mu$ and $\mu_2={\rm
Im}\,\mu$, we conclude by definition of $f_z$ and $f_{\bar z}$, see
(\ref{P}) and (\ref{D}), correspondingly, that the equality
(\ref{B}) at the point $z_0$ can be rewritten in the form
\[
(u_x+iv_x)+i(u_y+iv_y)=(\mu_1+i\mu_2)[(u_x+iv_x)-i(u_y+iv_y)]
\]
and, separating here the real and imaginary parts, we obtain the
system:
\[
u_x-v_y = \mu_1 (u_x + v_y) + \mu_2 (u_y - v_x)\ ,
\]
\[
u_y+v_x=\mu_1(v_x-u_y)+\mu_2(u_x+v_y)\ ,
\]
and then, grouping the terms containing $v$ on the left and the
terms containing $u$ on the right, and changing sometimes signs, we
have that:
\[
-\mu_2 v_x + \left( 1 + \mu_1 \right) v_y=\left( 1 - \mu_1 \right)
u_x - \mu_2 u_y\ ,
\]
\[
-(1-\mu_1) v_x  + \mu_2  v_y = -\mu_2 u_x + \left( 1 + \mu_1 \right)
u_y\ ,
\]
or the same in the matrix form
$$
M_1\,\nabla v= M_2\, \nabla u\ ,
$$
where
$$
M_1=
\begin{pmatrix}
-\mu_2 & (1+\mu_1) \\
-(1-\mu_1) & \mu_2
\end{pmatrix} \ ,      \    \ \ \ \ M_2=
\begin{pmatrix}
(1-\mu_1) & -\mu_2 \\
-\mu_2 & (1+\mu_1)
\end{pmatrix}\ .$$
Thus, the pair $(u,v)$ satisfies the equality (\ref{B}) at $z_0$
with $B:=M_1^{-1}M_2$. It is easy to see that ${\rm
det}\,M_1=1-|\mu|^2$, and hence directly to verify that
$$
M_1^{-1}=\frac{1}{1-|\mu|^2}\begin{pmatrix} \mu_2 & -(1+\mu_1)\\
(1-\mu_1)  & -\mu_2
\end{pmatrix}\  =\begin{pmatrix}
\frac{\mu_2}{1-|\mu|^2} & -\frac{1+\mu_1 }{1-|\mu|^2}\\
\frac{1-\mu_1}{1-|\mu|^2}  & -\frac{\mu_2}{1-|\mu|^2}
\end{pmatrix}\ .
$$
Finally, we obtain from here the equality (\ref{BM}) at $z_0$ with
$B$ in (\ref{Bmu}).\end{proof}

Let us denote by $\mathbb {\bf B^{2\times 2}}$ space of all
${2\times 2}$ matrices with real entries,
\begin{equation}\label{bB}
B\ =\ \begin{pmatrix}
b_{11} & b_{12} \\
b_{21} & b_{22}\
\end{pmatrix}\ ,
\end{equation}
with ${\rm det}\,B=1$, antisymmetric with respect to its auxiliary
diagonal, i.e., with $b_{22}=-b_{11}$, and with the {\bf ellipticity
condition} $|\mu|<1$, where
\begin{equation}\label{Mu}
\mu\ =\ \mu_B\ :=\
\frac{b_{12}+b_{21}-2\,i\,b_{11}}{b_{12}-b_{21}-2}\ .
\end{equation}


\begin{remark}\label{rem}  Theorem 16.1.6 in \cite{AIM} stated a one-to-one correspondence
between the coefficients $\mu$ in (\ref{B}) and $A$ in (\ref{DIV}).
Thus, the condition ${\rm det}\,B=1$ follows, for instance, by
remarks to this theorem applied to the affine mappings
$f(z)=f^{\mu}(z):=z+\mu \cdot\overline z$, $z=x+i\,y$, of the whole
plane with constant $\mu$ because of the relations (\ref{HA}) and
(\ref{HB}) and because of that ${\rm det}\,{\mathbb H}=1$. Moreover,
under ${\rm det}\,B=1$ and $b_{22}=-b_{11}$, it is easy to see that
the ellipticity condition $|\mu|^2<1$ is equivalent to the condition
$b_{21}>b_{12}$ and by Proposition \ref{pr} to the conditions
$b_{12}<0$ and $b_{21}>0$. \end{remark}

Thus, the properties ${\rm det}\,B=1$, $b_{22}=-b_{11}$, $b_{12}<0$
and $b_{21}>0$ give the {\bf complete characterization} of matrices
$B\in\mathbb B^{2\times 2}$ in the Cauchy-Riemann equations
(\ref{BM}) corresponding to the Beltrami equations (\ref{B}).

\medskip

\begin{proposition}\label{p} Let $D$ be a domain in $\mathbb R^2$, functions $u:D\to\mathbb R$
and $v:D\to\mathbb R$ have partial derivatives in $x$ and $y$ at a
point $Z_0=(x_0,y_0)\in\mathbb R^2$ in $D$. Suppose that the pair of
the functions $(u,v)$ satisfies the equality (\ref{BM}) at the point
$Z_0$ with $B\in\mathbb B^{2\times 2}$. Then the function $f:=u+iv$,
$z=x+i\,y$, satisfy the equality (\ref{B}) at the point
$z_0:=x_0+i\,y_0$ in $\mathbb C$ with the complex coefficient $\mu =
\mu_B$ given in (\ref{Mu}). \end{proposition}

Although later on we use only Proposition \ref{pr} and Remark
\ref{rem}, but Proposition \ref{p} has an independent interest for
the theory of the ge\-ne\-ra\-li\-zed Cauchy-Riemann equations
(\ref{BM}). Therefore, we give here its direct proof despite the
fact that it is perfectly routine and quite long.

\begin{proof} The equality (\ref{BM}) can be written as the system of
the equations
\begin{equation}\label{vx}
v_x\ =\ b_{11} u_x\ +\ b_{12} u_y
\end{equation}
\begin{equation}\label{vy}
v_y\ =\ b_{21} u_x\ -\ b_{11} u_y\ .
\end{equation}
Note also that it follows from the condition ${\rm det} B=1$ that
\begin{equation}\label{cond_det-1}
b^2_{11}\ +\ b_{12}b_{21}\ =\ -1\ .
\end{equation}

Next, we find that
$$
\mu\ :=\
\frac{f_{\overline{z}}}{f_{z}}=\frac{f_{x}+if_y}{f_{x}-if_y}=\frac{u_x+iv_x+i(u_y+iv_y)}{u_x+iv_x-i(u_y+iv_y)}=\frac{u_x-v_y
+i(v_x+u_y)}{u_x+v_y+i(v_x-u_y)}
$$
and, using the formulas (\ref{vx}) and  (\ref{vy}),  we get
$$
\mu\ =\ \frac{u_x-(b_{21} u_x-b_{11} u_y) +i((b_{11} u_x+ b_{12}
u_y)+u_y)}{u_x+(b_{21} u_x-b_{11} u_y)+i((b_{11} u_x+ b_{12}
u_y)-u_y)}\ =
$$
$$
=\ \frac{(1-b_{21})u_x+b_{11} u_y+i(b_{11} u_x+(1+b_{12} )u_y)
}{(1+b_{21})u_x-b_{11} u_y+i(b_{11} u_x+(b_{12}-1)u_y) }\ .
$$

Now, introducing the notations
$$\alpha=(1-b_{21})u_x+b_{11} u_y\ , \quad \quad \beta=b_{11} u_x+(1+b_{12} )u_y\ ,$$
and
$$\gamma=(1+b_{21})u_x-b_{11} u_y\ , \quad \quad \delta=b_{11} u_x+(b_{12}-1)u_y\ ,$$
we obtain the formula
$$
\mu\ =\ \frac{\alpha +i\beta}{\gamma +i\delta}\ =\ \frac{\alpha
\gamma +\beta\delta}{\gamma^2 +\delta^2}\ +\ i\ \frac{\beta
\gamma-\alpha \delta}{\gamma^2 +\delta^2}\ .
$$

Further, we calculate
$$\alpha\gamma\ =\ (1-b^2_{21})u^2_x\ +\ 2b_{11}b_{21}u_xu_y\ -\ b^2_{11}u^2_y\ ,$$
$$\beta\delta\ =\ b^2_{11}u^2_x\ +\ 2b_{11}b_{12}u_xu_y\ +\ (b^2_{12}-1)u^2_y\ ,$$
and
$$\beta\gamma=(b_{11}+b_{11}b_{21})u^2_x+(1+b_{12}+b_{21}+b_{12}b_{21}-b_{11}^2)u_xu_y-(b_{11}+b_{11}b_{12})u^2_y\ , $$
$$\alpha\delta=(b_{11}-b_{11}b_{21})u^2_x+(b_{11}^2+b_{12}+b_{21}-b_{12}b_{21}-1)u_xu_y+(b_{11}b_{12}-b_{11})u^2_y\ , $$
and, consequently,
$$\alpha \gamma+\beta \delta=(1-b^2_{21}+b^2_{11})u^2_x+2b_{11}(b_{21}+b_{12})u_xu_y+(b^2_{12}-b^2_{11}-1)u^2_y\ ,$$
$$\beta \gamma-\alpha \delta=2b_{11}b_{21}u^2_x+2(1+b_{12}b_{21}-b^2_{11})u_xu_y-2b_{11}b_{12}u^2_y\ , $$
$$\gamma^2+\delta^2=(1+b^2_{21}+2b_{21}+b^2_{11})u^2_x-2b_{11}(2+  b_{21}-b_{12}         )u_xu_y+(b^2_{11}+b^2_{12}-2b_{12}+1)u^2_y\ . $$

Moreover, in view of the condition (\ref{cond_det-1}), we have that
$$\alpha \gamma+\beta \delta=-b_{21}(b_{21}+b_{12})u^2_x+2b_{11}(b_{21}+b_{12})u_xu_y+b_{12}(b_{21}+b_{12})u^2_y=$$

$$=-(b_{21}+b_{12})(b_{21}u^2_x-2b_{11}u_xu_y-b_{12}u^2_y)\ ,$$

$$ \beta \gamma-\alpha \delta=2b_{11}b_{21}u^2_x-4b^2_{11}u_xu_y-2b_{11}b_{12}u^2_y= $$

$$=2b_{11}(b_{21}u^2_x-2b_{11}u_xu_y-b_{12}u^2_y)\ , $$

$$\gamma^2+\delta^2=b_{21}(b_{21}-b_{12}+2)u^2_x+2b_{11}(b_{12}-b_{21}-2)u_xu_y+b_{12}(b_{12}-b_{21}-2)u^2_y=$$

$$=(b_{12}-b_{21}-2)(-b_{21}u^2_x+2b_{11}u_xu_y+b_{12}u^2_y) =-(b_{12}-b_{21}-2)(b_{21}u^2_x-2b_{11}u_xu_y-b_{12}u^2_y)\ . $$

Thus,
$$
\mu\ =\ \mu_B\ =\
\frac{-(b_{21}+b_{12})(b_{21}u^2_x-2b_{11}u_xu_y-b_{12}u^2_y)}{-(b_{12}-b_{21}-2)(b_{21}u^2_x-2b_{11}u_xu_y-b_{12}u^2_y)}\
+\
$$
$$
+\
\frac{2ib_{11}(b_{21}u^2_x-2b_{11}u_xu_y-b_{12}u^2_y)}{-(b_{12}-b_{21}-2)(b_{21}u^2_x-2b_{11}u_xu_y-b_{12}u^2_y)}
$$
and, finally, we obtain from the latter the formula (\ref{Mu}).
\end{proof}

\section{On regular homeomorphic solutions}
Given a domain $D$ in $\mathbb R^2$, we say that a pair $(u,v)$ of
functions $u:D\to\mathbb R$ and $v:D\to\mathbb R$ in the class
$W^{1,1}_{\rm loc}$ is a {\bf regular homeomorphic solution} of the
generalized Cauchy-Riemann equation (\ref{BM})  in $D$ with a matrix
valued coefficient $B:D\to\mathbb M^{2\times 2}$ if $(u,v)$
satisfies (\ref{BM}) and $\nabla\, u\neq 0$ and $\nabla\, v\neq 0$
a.e. in $D$ and, moreover, the correspondence $(x,y)\mapsto(u,v)$ is
a homeomorphism (embedding) of $D$ into $\mathbb R^2$.

\medskip

In this section, we give consequences on regular homeomorphic
solutions for the generalized Cauchy-Riemann equation (\ref{BM})
based on a theo\-ry of regular homeomorphic solutions to the
Beltrami equations (\ref{B}).

\medskip

The {\bf dilatation quotient} of the equation (\ref{B}) is the
quantity
\begin{equation}\label{eqKPRS1.1}K_{\mu}(z)\ :=\ \frac{1+|\mu(z)|}{1-|\mu(z)|}\ .\end{equation}
The Beltrami equation is called {\bf degenerate} if ${\rm ess}\,{\rm
sup}\,K_{\mu}(z)=\infty$.

\medskip

It is known that if $K_{\mu}$ is bounded, then the Beltrami equation
has ho\-meo\-mor\-phic solutions with generalized derivatives by
Sobolev, see e.g. monographs \cite{Ah}, \cite{BGMR} and \cite{LV}.
Recently, a series of effective criteria for existence of
homeomorphic solutions have been also established for degenerate
Beltrami equations, see e.g. historic comments with relevant
references in monographs \cite{AIM}, \cite{GRSY} and \cite{MRSY}.

\medskip

These criteria were formulated both in terms of $K_{\mu}$ and the
more refined quantity that takes into account not only the modulus
of the complex coefficient $\mu$ but also its argument
\begin{equation}\label{eqTangent} K^T_{\mu}(z,z_0)\ :=\
\frac{\left|1-\frac{\overline{z-z_0}}{z-z_0}\mu (z)\right|^2}{1-|\mu
(z)|^2} \end{equation} that is called the {\bf tangent dilatation
quotient} of the Beltrami equation with respect to a point
$z_0\in\mathbb C$, see e.g. the articles \cite{AC$_2$},
\cite{GMR}--\cite{GRSY22+}, \cite{GRY}--\cite{ER},
\cite{KPR1}--\cite{KPRS}, \cite{Le}, \cite{RW} and
\cite{RSY1}--\cite{RSY10}. Note that
\begin{equation}\label{eqConnect} K^{-1}_{\mu}(z)\leqslant K^T_{\mu}(z,z_0) \leqslant K_{\mu}(z)
\ \ \ \ \ \ \ \forall\ z\in D\, ,\ z_0\in \Bbb C\ .\end{equation} A
geometrical sense of $K^T_{\mu}$ can be found in the monographs
\cite{GRSY} and \cite{MRSY}.

\medskip

Let $D$ be a domain in the complex plane ${\Bbb C}$. Recall that a
function $f:D\to\mathbb C$ in the Sobolev class $W^{1,1}_{\rm loc}$
is called a {\bf regular solution} of the Beltrami equation
(\ref{B}) in $D$ if $f$ satisfies (\ref{B}) a.e. in $D$ and its
Jacobian $J_f(z)=|f_z|^2-|f_{\bar z}|^2$ is positive a.e. in $D$.
Hereafter $d{\cal L}(z)$ corresponds to the Lebesgue measure.


\begin{lemma}\label{lem1} Let $D$ be a domain in the plane and let $B: D\to{\mathbb B}^{2\times 2}$
be a measurable function. Suppose that $K_{\mu_B}\in L_{\rm
loc}^1(D)$ and, for each $z_0\in D$,
\begin{equation}\label{eqT}
\int\limits_{{\varepsilon}<|z-z_0|<{\varepsilon}_0}\
K^T_{{\mu_B}}(z,z_0)\cdot{\psi}^2_{z_0,{\varepsilon}}(|z-z_0|)\
d{\cal L}(z)\ =\ o(I^2_{z_0}({\varepsilon})) \ \ \  \ \ \ \hbox{ as
${\varepsilon}\to 0$}\ ,\end{equation} where $\mu_B$ is given by the
formula (\ref{Mu}) and
${\psi}_{z_0,{\varepsilon}}:(0,\varepsilon_0)\to(0,\infty),$
${\varepsilon}\in(0,{\varepsilon}_0),$
$\varepsilon_0=\varepsilon(z_0)>0,$ is a family of measurable
functions such that
\begin{equation}\label{eqI}
I_{z_0}({\varepsilon})\ \colon =\
\int\limits_{{\varepsilon}}^{{\varepsilon}_0}{\psi}_{z_0,{\varepsilon}}(t)\
dt\ <\ \infty\ \ \ \ \ \ \forall\
{\varepsilon}\in(0,{\varepsilon}_0)\ .\end{equation}

Then the generalized Cauchy-Riemann equation (\ref{BM}) has a
regular homeomorphic solution $(u,v)$ such that $u={\rm Re}\,f$ and
$v={\rm Im}\,f$, where $f=f_{\mu_B}$ is a regular homeomorphic
solution of the Beltrami equation (\ref{B}) with $\mu=\mu_B$.

Furthermore, if $K_{\mu_B}\in L^1(D)$ and (\ref{eqT}) holds for all
$z_0\in \overline D$, then $(u,v)$ can be extended by continuity to
the whole plane as a regular homeomorphic solution of (\ref{BM})
with $B$ extended by Hodge matrix $\mathbb H$ outside~$D$.
\end{lemma}

\begin{proof}
Indeed, by Lemma 3 and Remark 2 in \cite{RSY6} the Beltrami equation
(\ref{B}) has a regular homeomorphic solution $f$ in $D$ under the
given hypotheses on $\mu :=\mu_B$ in (\ref{Mu}). By Proposition
\ref{pr}, see also Remark \ref{rem}, the pair $(u,v)$ of the
functions $u={\rm Re}\,f$ and $v={\rm Im}\,f$ is a solution of the
generalized Cauchy-Riemann equation (\ref{BM}). This solution is
regular because the Jacobian matrix of the function $f=u+iv$,
\begin{equation}\label{J}
\begin{pmatrix}
u_x & u_y  \\
v_x & v_y\
\end{pmatrix}\ ,
\end{equation}
interpreted as a mapping from $\mathbb R^2$ into $\mathbb R^2$ has
determinant zero at a point $z=x+iy$ in $D$ if either $\nabla
u(z)=0$ or $\nabla v(z)=0$ at the point. \end{proof}

\begin{remark}\label{rem1} Hereafter, we assume that $K^T_{\mu_B}(z,z_0)$ is extended by
$1$ outside~$D$. Note also that if either $K_{\mu_B}\in L^1(D)$ in a
bounded domain $D$ or $\mu_B$ has compact support, then the
homeomorphic solution $f_{\mu_B}$ of the Beltrami equation (\ref{B})
with $\mu=\mu_B$ extended by zero outside~$D$ can be chosen with the
hydro\-dynamic normalization: $f_{\mu_B}(z)=z+o(1)$ as $z\to\infty$,
see \cite{GRSY22} and \cite{GRSY22+}. \end{remark}

Applying interconnections between various integral conditions that
are relevant to (\ref{eqT}), we are able to obtain a number of
effective criteria for solvability, representation and regularity of
solutions for generalized Cauchy-Riemann equation (\ref{BM}).

\medskip

Recall first of all that a real-valued function $ \varphi $ in a
domain $D$ of ${\mathbb C}$ is called of {\bf bounded mean
oscillation} in $D$, abbr. $\varphi\in{\rm BMO}(D)$, if $\varphi\in
L_{\rm loc}^1(D)$ and
\begin{equation}\label{lasibm_2.2_1}\Vert \varphi\Vert_{*}:=
\sup\limits_{B}{\frac{1}{|B|}}\int\limits_{B}|\varphi(z)-\varphi_{B}|\,d{\cal
L}(z)<\infty\,,\end{equation} where the supremum is taken over all
discs $B$ in $D$ and
$$\varphi_{B}={\frac{1}{|B|}}\int\limits_{B}\varphi(z)\,d{\cal L}(z)\,.$$ We write $\varphi\in{\rm BMO}_{\rm loc}(D)$ if
$\varphi\in{\rm BMO}(U)$ for each relatively compact subdomain $U$
of $D$. We also write sometimes for short BMO and ${\rm BMO}_{\rm
loc }$, respectively. Finally, we write $\varphi\in{\rm
BMO}(\overline D)$ if $\varphi_*\in{\rm BMO}(D_*)$ for some
extension $\varphi_*$ of the function $\varphi$ into a domain $D_*$
containing $\overline D$.

The class BMO was introduced by John and Nirenberg (1961) in the
paper \cite{JN} and soon became an important concept in harmonic
analysis, partial diffe\-rential equations and related areas, see
e.g. articles \cite{AIKM}, \cite{BN}, \cite{GRSY22}, \cite{IKM},
\cite{MRV}, \cite{MU}, \cite{RSY9}, \cite{RSY10} and monographs
\cite{HKM} and \cite{RR}.


A function $\varphi$ in BMO is said to have {\bf vanishing mean
oscillation}, abbr. $\varphi\in{\rm VMO}$, if the supremum in
(\ref{lasibm_2.2_1}) taken over all balls $B$ in $D$ with
$|B|<\varepsilon$ converges to $0$ as $\varepsilon\to0$. VMO has
been introduced by Sarason in \cite{Sar}. There are a number of
papers devoted to the study of partial differential equations with
coefficients of the class VMO, see e.g. \cite{CFL}, \cite{IS},
\cite{Pal}, \cite{Ra$_1$} and \cite{Ra$_2$}.

\begin{remark}\label{rem2} Note that $W^{\,1,2}\left({{D}}\right) \subset VMO
\left({{D}}\right),$ see e.g. \cite{BN}. \end{remark}

Following \cite{IR}, we say that a locally integrable function
$\varphi:D\to{\Bbb R}$ has {\bf finite mean oscillation} at a point
$z_0\in D$, abbr. $\varphi\in{\rm FMO}(z_0)$, if
\begin{equation}\label{FMO_eq2.4}\overline{\lim\limits_{\varepsilon\to0}}\ \ \
\dashint_{B(x_0,\varepsilon)}|{\varphi}(z)-\widetilde{\varphi}_{\varepsilon}(z_0)|\,d{\cal
L}(z)<\infty\,,\end{equation} where
\begin{equation}\label{FMO_eq2.5}
\widetilde{\varphi}_{\varepsilon}(z_0)=\dashint_{B(z_0,\varepsilon)}
{\varphi}(z)\,d{\cal L}(z)\end{equation} is the mean value of the
function ${\varphi}(z)$ over disk $B(z_0,\varepsilon):=\{
z\in{\mathbb C}: |z-z_0|<\varepsilon\}$. Thereafter we assume that
$\varphi$ is integrable ar least in a neighborhood of the point
$z_0$.

\medskip

The following statement is obvious by the triangle inequality.

\begin{proposition}\label{pr1}
If, for a  collection of numbers $\varphi_{\varepsilon}\in{\Bbb R}$,
$\varepsilon\in(0,\varepsilon_0]$,
\begin{equation}\label{FMO_eq2.7}\overline{\lim\limits_{\varepsilon\to0}}\ \ \
\dashint_{B(z_0,\varepsilon)}|\varphi(z)-\varphi_{\varepsilon}|\,d{\cal
L}(z)<\infty\,,\end{equation} then $\varphi $ is of finite mean
oscillation at $z_0$.
\end{proposition}

In particular, choosing here  $\varphi_{\varepsilon}\equiv0$,
$\varepsilon\in(0,\varepsilon_0]$ in Proposition~\ref{pr1}, we
obtain the following.

\begin{corollary}\label{cor1}
If, for a point $z_0\in D$,
\begin{equation}\label{FMO_eq2.8}\overline{\lim\limits_{\varepsilon\to 0}}\ \ \
\dashint_{B(z_0,\varepsilon)}|\varphi(z)|\,d{\cal L}(z)\ < \infty\,,
\end{equation} then $\varphi$ has finite mean oscillation at
$z_0$. \end{corollary}

Recall that a point $z_0\in D$ is called a {\bf Lebesgue point} of a
function $\varphi:D\to{\Bbb R}$ if $\varphi$ is integrable in a
neighborhood of $z_0$ and \begin{equation}\label{FMO_eq2.7a}
\lim\limits_{\varepsilon\to 0}\ \ \ \dashint_{B(z_0,\varepsilon)}
|\varphi(z)-\varphi(z_0)|\,d{\cal L}(z)=0\,.\end{equation} It is
known that, almost every point in $D$ is a Lebesgue point for every
function $\varphi\in L_1(D)$. Thus, we have by Proposition~\ref{pr1}
the next corollary.

\begin{corollary}\label{cor2}
Every locally integrable function $\varphi:D\to{\Bbb R}$ has a
finite mean oscillation at almost every point in $D$.
\end{corollary}

\begin{remark}\label{rem3}
Note that the function $\varphi(z)=\log\left(1/|z|\right)$ belongs
to BMO in the unit disk $\Delta$, see e.g. \cite{RR}, p. 5, and
hence also to FMO. However,
$\widetilde{\varphi}_{\varepsilon}(0)\to\infty$ as
$\varepsilon\to0$, showing that condition (\ref{FMO_eq2.8}) is only
sufficient but not necessary for a function $\varphi$ to be of
finite mean oscillation at $z_0$. Clearly, ${\rm BMO}(D)\subset{\rm
BMO}_{\rm loc}(D)\subset{\rm FMO}(D)$ and as well-known ${\rm
BMO}_{\rm loc}\subset L_{\rm loc}^p$ for all $p\in[1,\infty)$, see
e.g. \cite{JN} or \cite{RR}. However, FMO is not a subclass of
$L_{\rm loc}^p$ for any $p>1$ but only of $L_{\rm loc}^1$. Thus, the
class FMO is much more wider than ${\rm BMO}_{\rm loc}$.
\end{remark}

\medskip

Versions of the next lemma has been first proved for the class BMO
in \cite{RSY10}. For the FMO case, see papers \cite{IR}, \cite{RS},
\cite{RSY5}, \cite{RSY8}  and monographs \cite{GRSY} and
\cite{MRSY}.

\medskip
\begin{lemma}\label{lem2} Let $D$ be a domain in ${\Bbb C}$ and let
$\varphi:D\to{\Bbb R}$ be a  non-negative function  of the class
${\rm FMO}(z_0)$ for some $z_0\in D$. Then
\begin{equation}\label{eq13.4.5}\int\limits_{\varepsilon<|z-z_0|<\varepsilon_0}\frac{\varphi(z)\,d{\cal L}(z)}
{\left(|z-z_0|\log\frac{1}{|z-z_0|}\right)^n}=O\left(\log\log\frac{1}{\varepsilon}\right)\
\quad\text{as}\quad\varepsilon\to 0\end{equation} for some
$\varepsilon_0\in(0,\delta_0)$ where $\delta_0=\min({\rm
exp}\,{-e},d_0)$, $d_0=\sup\limits_{z\in D}|z-z_0|$.
\end{lemma}

Choosing $\psi(t)=1/\left(t\, \log\left(1/t\right)\right)$ in
Lemma~\ref{lem1}, we obtain by Lem\-ma~\ref{lem2} the following
result with the FMO type criterion.

\begin{theorem}\label{th1} Let a function $B: D\to{\mathbb B}^{2\times 2}$
be measurable and let $\mu_B$ be given by the formula (\ref{Mu}).
Suppose that $K_{\mu_B}\in L_{\rm loc}^1(D)$ and
$K^T_{\mu_B}(z,z_0)\leqslant Q_{z_0}(z)$ a.e. in a neighborhood
$U_{z_0}$ of each point $z_0\in D$ for a function $Q_{z_0}:
U_{z_0}\to[1,\infty]$ in class ${\rm FMO}({z_0})$.

Then the generalized Cauchy-Riemann equation (\ref{BM}) has a
regular homeomorphic solution $(u,v)$ such that $u={\rm Re}\,f$ and
$v={\rm Im}\,f$, where $f=f_{\mu_B}$ is a regular homeomorphic
solution of Beltrami equation (\ref{B}) with $\mu=\mu_B$.

Furthermore, if $K_{\mu_B}\in L^1(D)$ and the hypotheses hold for
all $z_0\in \overline D$, then $(u,v)$ can be extended by continuity
to the whole plane as a regular homeomorphic solution of (\ref{BM})
with $B$ extended by the Hodge matrix $\mathbb H$ outside~$D$.
\end{theorem}

By Corollary~\ref{cor1} we obtain the following nice consequence of
Theorem~\ref{th1}, where $B(z_0,\varepsilon)$ denote the disks $\{ z
\in\mathbb C:\, |z-z_0|\,<\, \varepsilon\}$.

\begin{corollary}\label{cor3} Let a function $B: D\to{\mathbb B}^{2\times 2}$
be measurable and let $\mu_B$ be given by the formula (\ref{Mu}).
Suppose that $K_{\mu_B}\in L_{\rm loc}^1(D)$ and that
\begin{equation}\label{eqMEAN}\overline{\lim\limits_{\varepsilon\to0}}\quad
\dashint_{B(z_0,\varepsilon)}K^T_{\mu_B}(z,z_0)\,d{\cal L}(z)\ <\
\infty\ \ \ \ \ \ \ \ \forall\ z_0\in D\ .\end{equation}

Then the generalized Cauchy-Riemann equation (\ref{BM}) has a
regular homeomorphic solution $(u,v)$ such that $u={\rm Re}\,f$ and
$v={\rm Im}\,f$, where $f=f_{\mu_B}$ is a regular homeomorphic
solution of Beltrami equation (\ref{B}) with $\mu=\mu_B$.

Furthermore, if $K_{\mu_B}\in L^1(D)$ and (\ref{eqT}) holds for all
$z_0\in \overline D$, then $(u,v)$ can be extended by continuity to
the whole plane as a regular homeo\-morphic solution of (\ref{BM})
with $B$ extended by the matrix $\mathbb H$ outside~$D$.
\end{corollary}

\medskip

We also obtain the following consequences of Theorem~\ref{th1} with
the BMO type criteria.

\begin{corollary}\label{cor4} Let a function $B: D\to{\mathbb
B}^{2\times 2}$ be measurable and let $\mu_B$ be given by the
formula (\ref{Mu}). If $K_{\mu_B}$ has a dominant $Q:
D\to[1,\infty]$ in the class ${\rm BMO}(D)$, then the first part of
the conclusions of Theorem~\ref{th1} hold. Furthermore, if
$K_{\mu_B}$ has a dominant $Q$ in the class ${\rm BMO}(\overline
D)$, then the second part of the conclusions of Theorem~\ref{th1}
hold.
\end{corollary}

\begin{remark}\label{rem4} In particular, the first part of conclusions of Theorem~\ref{th1} hold if
$Q\in{\rm W}^{1,2}_{\rm loc}(D)$ because of $W^{\,1,2}_{\rm loc}(D)
\subset {\rm VMO}_{\rm loc}(D)$, see Remark~\ref{rem2}. The second
part holds if $Q\in{\rm W}^{1,2}_{\rm loc}(\overline D)$.
\end{remark}

\medskip

Choosing in Lemma~\ref{lem1} the functional parameter $\psi(t)=1/t$,
we come to the next statement with the Calderon-Zygmund type
criterion.

\begin{theorem}\label{th2} Let a function $B: D\to{\mathbb B}^{2\times
2}$ be measurable and let $\mu_B$ be given by the formula
(\ref{Mu}). Suppose that $K_{\mu_B}\in L_{\rm loc}^1(D)$ and that,
for each $z_0\in D$ and some $\varepsilon_0=\varepsilon(z_0)>0$,
\begin{equation}\label{eqLOG}
\int\limits_{\varepsilon<|z-z_0|<\varepsilon_0}\frac{K^T_{\mu_B}(z,z_0)}{|z-z_0|^2}\
{d{\cal L}(z)}\ =\
o\left(\left[\log\frac{1}{\varepsilon}\right]^2\right)\ \ \ \hbox{
as $\varepsilon\to 0$}\ .\end{equation}

Then the generalized Cauchy-Riemann equation (\ref{BM}) has a
regular homeomorphic solution $(u,v)$ such that $u={\rm Re}\,f$ and
$v={\rm Im}\,f$, where $f=f_{\mu_B}$ is a regular homeomorphic
solution of Beltrami equation (\ref{B}) with $\mu=\mu_B$.

Furthermore, if $K_{\mu_B}\in L^1(D)$ and (\ref{eqT}) holds for all
$z_0\in \overline D$, then $(u,v)$ can be extended by continuity to
the whole plane as a regular homeomorphic solution of (\ref{BM})
with $B$ extended by the matrix $\mathbb H$ outside~$D$.
\end{theorem}

\begin{remark}\label{rem5}  Choosing in Lemma~\ref{lem1} the functional parameter
$\psi(t)=1/(t\log{1/t})$ instead of $\psi(t)=1/t$, we are able to
replace (\ref{eqLOG}) by the conditions
\begin{equation}\label{eqLOGLOG}
\int\limits_{\varepsilon<|z-z_0|<\varepsilon_0}\frac{K^T_{\mu_B}(z,z_0)\
d{\cal L}(z)} {\left(|z-z_0|\log{\frac{1}{|z-z_0|}}\right)^2}\ =\
o\left(\left[\log\log\frac{1}{\varepsilon}\right]^2\right)\ \ \ \ \
 \forall\ z_0\in D\end{equation} as $\varepsilon\to 0$ for some $\varepsilon_0=\varepsilon(z_0)>0$.
More generally, we are able to give here the whole scale of the
corresponding conditions with ite\-rated logarithms, i.e., using
functions
$\psi(t)=1/(t\log{1}/{t}\cdot\log\log{1}/{t}\cdot\ldots\cdot\log\ldots\log{1}/{t})$.
\end{remark}

Similarly, choosing in Lemma~\ref{lem1} the corresponding functional
parameter $\psi(t)$, we obtain the following Lehto type criterion,
cf. his paper \cite{Le} that was devoted to the Beltrami equations
on the plane. In this part, $k^T_{\mu_B}(z_0,r)$ is $L^1-$norm of
the tangent dilatation quotient  $K^T_{\mu_B}(z,z_0)$ of the
Beltrami equation (\ref{B}) with $\mu=\mu_B$ over the circle
$C(z_0,r):=\{ z\in\mathbb C:\ |z-z_0|=r\}$.

\begin{theorem}\label{th3} Let $B: D\to{\mathbb
B}^{2\times 2}$ be measurable and let $\mu_B$ be given by the
formula (\ref{Mu}). Suppose that $K_{\mu_B}\in L_{\rm loc}^1(D)$ and
\begin{equation}\label{eqLEHTO}\int\limits_{0}^{\varepsilon(z_0)}
\frac{dr}{k^T_{\mu_B}(z_0,r)}\ =\ \infty\ \ \ \ \ \ \ \ \ \ \forall\
z_0\in D\ .\end{equation}

Then the generalized Cauchy-Riemann equation (\ref{BM}) has a
regular homeomorphic solution $(u,v)$ such that $u={\rm Re}\,f$ and
$v={\rm Im}\,f$, where $f=f_{\mu_B}$ is a regular homeomorphic
solution of Beltrami equation (\ref{B}) with $\mu=\mu_B$.

Furthermore, if $K_{\mu_B}\in L^1(D)$ and (\ref{eqT}) holds for all
$z_0\in \overline D$, then $(u,v)$ can be extended by continuity to
the whole plane as a regular homeomorphic solution of (\ref{BM})
with $B$ extended by the matrix $\mathbb H$ outside~$D$.
\end{theorem}


Indeed, setting ${\psi}_0(t) : = 1/k^T_{\mu_B}(z_0,t)$, by the
Fubini theorem, see e.g. Theo\-rem III(9.3) in \cite{Saks}, by
elementary calculations and by (\ref{eqLEHTO}), we derive, for
$\varepsilon_0=\varepsilon(z_0)$, that
\begin{equation}\label{eq1000}
\int\limits_{\varepsilon<|z-z_0|<\varepsilon_0}K^T_{\mu_B}(z,z_0)\,\psi_0^2(|z-z_0|)\
d{\cal L}(z) = I_0(\varepsilon)
 = o\left(I_0^2(\varepsilon)\right)\ \ \ \hbox{as $\varepsilon\to 0$,}
\end{equation} where
$I_0(\varepsilon)=
\int\limits_{\varepsilon}^{\varepsilon_0}\psi_0(t)\,dt$,
$\varepsilon\in(0,\varepsilon_0)$. Thus, Theorem~\ref{th3} follows
from Lemma~\ref{lem1}.

\begin{corollary}\label{cor5} Let a function $B: D\to{\mathbb B}^{2\times 2}$ be
measurable and let $\mu_B$ be given by the formula (\ref{Mu}).
Suppose that $K_{\mu_B}\in L_{\rm loc}^1(D)$ and
\begin{equation}\label{eqLOGk}k^T_{\mu_B}(z_0,\varepsilon)\ =\ O\left(\log\frac{1}{\varepsilon}\right)
\qquad\mbox{as}\ \varepsilon\to0\ \ \ \ \ \ \ \  \forall\ z_0\in D\
.\end{equation}

Then the generalized Cauchy-Riemann equation (\ref{BM}) has a
regular homeomorphic solution $(u,v)$ such that $u={\rm Re}\,f$ and
$v={\rm Im}\,f$, where $f=f_{\mu_B}$ is a regular homeomorphic
solution of Beltrami equation (\ref{B}) with $\mu=\mu_B$.

Furthermore, if $K_{\mu_B}\in L^1(D)$ and (\ref{eqT}) holds for all
$z_0\in \overline D$, then $(u,v)$ can be extended by continuity to
the whole plane as a regular homeomorphic solution of (\ref{BM})
with $B$ extended by the matrix $\mathbb H$ outside~$D$.
 \end{corollary}

\begin{remark}\label{rem6} In particular, the conclusions of Theorem~\ref{th3} hold if
\begin{equation}\label{eqLOGK} K^T_{\mu_B}(z,z_0)\ =\ O\left(\log\frac{1}{|z-z_0|}\right)\qquad{\rm
as}\quad z\to z_0\ \ \ \ \ \ \ \ \forall\ z_0\in D\ .\end{equation}
Moreover, the condition (\ref{eqLOGk}) can be replaced by the series
of weaker conditions
\begin{equation}\label{edLOGLOGk}
k^T_{\mu_B}(z_0,\varepsilon)=O\left(\left[\log\frac{1}{\varepsilon}\cdot\log\log\frac{1}
{\varepsilon}\cdot\ldots\cdot\log\ldots\log\frac{1}{\varepsilon}
\right]\right)\ .
\end{equation}
\end{remark}

To get other type criteria, we need more a couple of auxiliary
statements.

\medskip

Further we use the following notions of the inverse function for
arbitrary monotone functions. Namely, for every non-decreasing
function $\Phi:[0,\infty]\to[0,\infty]$ the inverse function
$\Phi^{-1}:[0,\infty]\to[0,\infty]$ can be well-defined by setting
\begin{equation}\label{eqINVERSE}
\Phi^{-1}(\tau)\ :=\ \inf\limits_{\Phi(t)\geqslant\tau} t
\end{equation}
Here $\inf$ is equal to $\infty$ if the set of $t\in[0,\infty]$ such
that $\Phi(t)\geqslant\tau$ is empty. Note that the function
$\Phi^{-1}$ is non-decreasing, too. It is also evident immediately
by the definition that $\Phi^{-1}(\Phi(t)) \leqslant t$ for all
$t\in[0,\infty]$ with the equality except intervals of constancy of
the function $\Phi(t)$.

\medskip

First of all, recall equivalence of integral conditions, see
Theorem~2.1 in \cite{RSY3}.

\begin{proposition}\label{pr2} Let $\Phi:[0,\infty]\to[0,\infty]$ be a
non-decreasing function and set
\begin{equation}\label{eqLOGFi}
H(t)\ =\ \log\Phi(t)\ .
\end{equation}
Then the equality
\begin{equation}\label{eq333Frer}\int\limits_{\Delta}^{\infty}H'(t)\,\frac{dt}{t}=\infty,
\end{equation}
implies the equality
\begin{equation}\label{eq333F}\int\limits_{\Delta}^{\infty}
\frac{dH(t)}{t}=\infty\,,\end{equation} and (\ref{eq333F}) is
equivalent to
\begin{equation}\label{eq333B}
\int\limits_{\Delta}^{\infty}H(t)\,\frac{dt}{t^2}=\infty\,\end{equation}
for some $\Delta>0$, and (\ref{eq333B}) is equivalent to each of the
equalities
\begin{equation}\label{eq333C} \int\limits_{0}^{\delta}H\left(\frac{1}{t}\right)\,{dt}=\infty\end{equation} for
some $\delta>0$, \begin{equation}\label{eq333D}
\int\limits_{\Delta_*}^{\infty}\frac{d\eta}{H^{-1}(\eta)}=\infty\end{equation}
for some $\Delta_*>H(+0)$ and to
\begin{equation}\label{eq333a}
\int\limits_{\delta_*}^{\infty}\frac{d\tau}{\Phi^{-1}(\tau)}\ =\
\infty\end{equation} for some $\delta_*>\Phi(+0)$.

\medskip

Moreover, (\ref{eq333Frer}) is equivalent to (\ref{eq333F}) and to
hence (\ref{eq333Frer})–(\ref{eq333a}) are equivalent to each other
if $\Phi$ is in addition absolutely continuous. In particular, all
the given conditions are equivalent if $\Phi$ is convex and
non-decreasing. \end{proposition}

\medskip

Note that the integral in (\ref{eq333F}) is understood as the
Lebesgue--Stieltjes integral and the integrals in (\ref{eq333Frer})
and (\ref{eq333B})--(\ref{eq333D}) as the ordinary Lebesgue
integrals. It is necessary to give one more explanation. From the
right hand sides in the conditions (\ref{eq333Frer})--(\ref{eq333D})
we have in mind $+\infty$. If $\Phi(t)=0$ for $t\in[0,t_*$, then
$H(t)=-\infty$ for $t\in[0,t_*]$ and we complete the definition
$H'(t)=0$ for $t\in[0,t_*]$. Note, the conditions (\ref{eq333F}) and
(\ref{eq333B}) exclude that $t_*$ belongs to the interval of
integrability because in the contrary case the left hand sides in
(\ref{eq333F}) and (\ref{eq333B}) are either equal to $-\infty$ or
indeterminate. Hence we may assume in
(\ref{eq333Frer})--(\ref{eq333C}) that $\delta >t_0$,
correspondingly, $\Delta<1/t_0$ where
$t_0:=\sup\limits_{\Phi(t)=0}t$, and set $t_0=0$ if $\Phi(0)>0$.

\medskip

The most interesting condition (\ref{eq333B}) can be written in the
form:
\begin{equation}\label{eq5!}
\int\limits_{\Delta}^{\infty}\log\, \Phi(t)\ \frac{dt}{t^{2}}\ =\
+\infty\ \ \ \ \ \ \mbox{for some $\Delta > 0$}\ .
\end{equation}

\noindent The next follows by Corollary 3.2 in \cite{RSY3} with
$\lambda =1$ and Proposition~\ref{pr2}.

\begin{proposition}\label{pr3} Let $D$ be a domain in $\mathbb C$ and let
$Q:D\to[1,\infty]$ be a measurable function such that
\begin{equation}\label{eq5.555} \int\limits_D\Phi(Q(z))\,d{\cal L}(z)\ <\ \infty\ ,\end{equation} where
$\Phi:[0,\infty]\to[0,\infty]$ is a non-decreasing convex function
with (\ref{eq5!}). Then
\begin{equation}\label{eq3.333A}
\int\limits_0^{\varepsilon_0}\frac{dr}{q(z_0,r)}\ =\ \infty\ \ \ \ \
\ \ \forall\ z_0\in D\ ,\ \varepsilon_0\in(0,{\rm
dist}\,(z_0,\partial D))\ ,\end{equation} where $q(z_0,r)$ is
$L^1-$norm of the function $Q(z)$ over the circle $\mathbb
S(z_0,r)$.
\end{proposition}

Combining Proposition~\ref{pr3} with Theorem~\ref{th3} we obtain the
following significant results with the Orlicz type criteria.

\begin{theorem}\label{th4} Let a function $B: D\to{\mathbb B}^{2\times
2}$ be measurable and let $\mu_B$ be given by the formula
(\ref{Mu}). Suppose that
\begin{equation}\label{eqINTEG}
\int\limits_{D}\Phi\left(K_{\mu_B}(z)\right)\,d{\cal L}(z)\ <\
\infty
\end{equation} for a convex non-decreasing function $\Phi:(0,\infty]\to(0,\infty]$
satisfying the condition (\ref{eq5!}).

Then the generalized Cauchy-Riemann equation (\ref{BM}) has a
regular homeomorphic solution $(u,v)$ such that $u={\rm Re}\,f$ and
$v={\rm Im}\,f$, where $f=f_{\mu_B}$ is a regular homeomorphic
solution of Beltrami equation (\ref{B}) with $\mu=\mu_B$.

Furthermore, if $K_{\mu_B}\in L^1(D)$ and (\ref{eqT}) holds for all
$z_0\in \overline D$, then $(u,v)$ can be extended by continuity to
the whole plane as a regular homeomorphic solution of (\ref{BM})
with $B$ extended by the matrix $\mathbb H$ outside~$D$.
\end{theorem}

\begin{remark}\label{rem7} By Theorem~5.1 in \cite{RSY2}, condition (\ref{eq5!}) is not only
sufficient but also necessary to have the conclusions of
Theorem~\ref{th4} to all generalized Cauchy-Riemann equations
(\ref{BM}) with the integral restriction (\ref{eqINTEG}).
\end{remark}

In particular, choosing here $\ \Phi(t)\ :=\ {\rm exp}\,{\alpha t}\
$ with $\ \alpha
> 0\ $, we obtain the following.

\begin{corollary}\label{cor6} Let a function $B: D\to{\mathbb B}^{2\times
2}$ be measurable and let $\mu_B$ be given by the formula
(\ref{Mu}). Suppose that, for some $\alpha>0$,
\begin{equation}\label{eqEXP}\int\limits_{D}{\rm exp}\,{\alpha
K_{\mu_B}(z)}\ d{\cal L}(z)\ <\ \infty\ .
\end{equation}

Then the generalized Cauchy-Riemann equation (\ref{BM}) has a
regular homeomorphic solution $(u,v)$ such that $u={\rm Re}\,f$ and
$v={\rm Im}\,f$, where $f=f_{\mu_B}$ is a regular homeomorphic
solution of Beltrami equation (\ref{B}) with $\mu=\mu_B$.

Furthermore, if $K_{\mu_B}\in L^1(D)$ and (\ref{eqT}) holds for all
$z_0\in \overline D$, then $(u,v)$ can be extended by continuity to
the whole plane as a regular homeo\-morphic solution of (\ref{BM})
with $B$ extended by the matrix $\mathbb H$ outside~$D$.
\end{corollary}

\section{On the Dirichlet problem with continuous data}

In the previous section, we formulated a series of criteria for
existence of regular solutions $(u,v)$ to the generalized
Cauchy-Riemann equations (\ref{BM}) which generated a homeomorphic
correspondence $(x,y)\mapsto(u,v)$ in domains of the plane $\mathbb
R^2$. If we begin to consider a boundary value problem, then the
request on the homeomorphism becomes redundant and the problem
becomes generally unsolvable.

\medskip

Let us consider the {\bf Dirichlet problem} for the generalized
Cauchy-Riemann equations (\ref{BM}) consisting in finding its
solutions $(u,v)$ with prescribed continuous data $\varphi:\partial
D\to\mathbb R$ of potential $u$ at the boundary
\begin{equation}\label{eqDIR}
\lim\limits_{z\to\zeta}u(z)\ =\ \varphi(\zeta)\qquad\forall\
\zeta\in\partial D
\end{equation}
in arbitrary bounded simply connected domains $D$ in $\mathbb R^2$.

\medskip

Given a simply connected domain $D$ in $\mathbb R^2$, we say that a
pair $(u,v)$ of continuous functions $u:D\to\mathbb R$ and
$v:D\to\mathbb R$ in the class $W^{1,1}_{\rm loc}$ is a {\bf regular
solution of the Dirichlet problem} (\ref{eqDIR}) for the generalized
Cauchy-Riemann equation (\ref{BM}) in $D$ if $(u,v)$ satisfies
(\ref{BM}) a.e. in $D$ and, moreover, $\nabla\, u\neq 0$ and
$\nabla\, v\neq 0$ a.e. in $D$, and the correspondence
$(x,y)\mapsto(u,v)$ is a discrete and open mapping of $D$ into
$\mathbb R^2$.

\medskip

Recall that a mapping of a domain $D$ in $\mathbb C$ ($\mathbb R^2$)
into $\mathbb C$ ($\mathbb R^2$)  is called {\bf discrete} if the
preimage of each point in $\mathbb C$ ($\mathbb R^2$) consists of
isolated points in $D$ and {\bf open} if the mapping maps every open
set in $D$ onto an open set in $\mathbb C$ ($\mathbb R^2$). Note
that continuity, discreteness and openness are {\bf characteristic
topological properties of analytic functions} because each analytic
function has these properties and, moreover, by the known Stoilow
theorem, see e.g. \cite{Sto}, each function in the complex plane
domain with the given properties has a representation as a
composition of some analytic function and a homeomorphism.

\medskip

We apply here as a basis the article \cite{GRSY22}, where it was
proved a general lemma on the existence of the regular solutions $f$
of the corresponding Dirichlet problem with continuous boundary data
\begin{equation}\label{eqDIRICHLET}\lim\limits_{z\to\zeta}{\rm
Re}\,f(z)=\varphi(\zeta)\qquad\forall\ \zeta\in\partial
D\end{equation} for the Beltrami equations (\ref{B}) in arbitrary
bounded simply connected domains $D$ in $\mathbb C$. Namely, arguing
similarly to the proof of Lemma~\ref{lem1}, we obtain the following
general lemma on the basis of Lemma 4 in \cite{GRSY22}.

\begin{lemma}\label{lem3} Let $D$ be a bounded simply connected domain in $\mathbb
R^2$, let $B:D\to\mathbb B^{2\times 2}$ be a measurable function in
$D$ and let $\mu=\mu_B$ be given by the formula (\ref{Mu}) with
$K_{\mu_B}\in L^1(D)$. Suppose also that
\begin{equation}\label{3omalMA}
\int\limits_{\varepsilon<|z-z_0|<\varepsilon_0}
K^T_{\mu_B}(z,z_0)\cdot\psi^2_{z_0,\varepsilon}(|z-z_0|)\,d{\cal
L}(z)=o(I_{z_0}^{2}(\varepsilon))\quad\ \ \forall\
z_0\in\overline{D}\end{equation} as $\varepsilon\to0$ for some
$\varepsilon_0=\varepsilon(z_0)>0$ and a family of measurable
functions $\psi_{z_0,\varepsilon}: (0,\varepsilon_0)\to(0,\infty)$
such that
\begin{equation}\label{eq3.5.3MA}
I_{z_0}(\varepsilon)\colon
=\int\limits_{\varepsilon}^{\varepsilon_0}
\psi_{z_0,\varepsilon}(t)\,dt<\infty\qquad\forall\
\varepsilon\in(0,\varepsilon_0)\,.\end{equation}

Then the generalized Cauchy-Riemann equation (\ref{BM}) has a
regular solution $(u,v)$ of the Dirichlet problem (\ref{eqDIR}) in
$D$ for each continuous inconstant boundary data $\varphi:\partial
D\to{\Bbb R}$ such that $u={\rm Re}\,f$ and $v={\rm Im}\,f$, where
$f$ is a regular solution of the Dirichlet problem
(\ref{eqDIRICHLET}) for the Beltrami equation (\ref{B}) with
$\mu=\mu_B$.
\end{lemma}

As it was before, we assume here and further that the dilatation
quotients $K_{\mu}(z)$ and $K^T_{\mu}(z,z_0)$ of the Beltrami
equation (\ref{B}) with the complex coefficient $\mu=\mu_B$ are
extended to the whole complex plane $\mathbb C$ by $1$ outside the
domain $D$.

\begin{remark}\label{remH}
Note also that in turn $f$ has the representation as the composition
${\cal A}\circ g|_D$, where $g:\mathbb C\to\mathbb C$ is a regular
homeomorphic solution of the Beltrami equation (\ref{B}) in $\mathbb
C$ with $\mu=\mu_B$ extended by zero outside~$D$ with hydrodynamic
normalization at infinity and $\cal A$ is an analytic function in
the domain $D_*:=g(D)$, and
\begin{equation}\label{eqHYDROmA}
u\ =\ {\cal H}\circ g|_D\ ,\ \ \ \ \ \ \hbox{$g(z)\, =\ z\, +\,
o(1)$\ \ \ as\ \ \ $z\to\infty$}\ ,
\end{equation}
${\cal H}:D_*\to\mathbb C$  is the unique harmonic function with the
Dirichlet condition
\begin{equation}\label{eqHOLOMORPHICMA}
\lim_{\xi\to\zeta}\ {\cal H}(\xi)\ =\ \varphi_*(\zeta)\ \ \ \ \
\forall\ \zeta\in\partial D_*\ ,\ \ \ \ \ \mbox{where
$\varphi_*:=\varphi\circ g^{-1}$.}
\end{equation}
\end{remark}

Note, the main difference of the integral conditions (\ref{3omalMA})
in Lemma \ref{lem3} and others in the following consequences that
they must take place here for all boundary points also but not only
for inner points.

\medskip

Next, arguing similarly to Section 3, we derive from
Lemma~\ref{lem3} the following series of interesting consequences on
the Dirichlet problem (\ref{eqDIR}) for the generalized
Cauchy-Riemann equations (\ref{BM}).

\begin{theorem}\label{th5} Let $D$ be a bounded simply connected domain in $\mathbb
R^2$, let $B:D\to\mathbb B^{2\times 2}$ be a measurable function in
$D$ and let $\mu=\mu_B$ be given by the formula (\ref{Mu}) with
$K_{\mu_B}\in L^1(D)$. Suppose that $K^T_{\mu_B}(z,z_0)\leqslant
Q_{z_0}(z)$ a.e. in $U_{z_0}$ for every point $z_0\in \overline{D}$,
a neighborhood $U_{z_0}$ of $z_0$ and a function $Q_{z_0}:
U_{z_0}\to[0,\infty]$ in the class ${\rm FMO}({z_0})$.

Then the generalized Cauchy-Riemann equation (\ref{BM}) has a
regular solution $(u,v)$ of the Dirichlet problem (\ref{eqDIR}) in
$D$ for each continuous inconstant boundary data $\varphi:\partial
D\to{\Bbb R}$ such that $u={\rm Re}\,f$ and $v={\rm Im}\,f$, where
$f$ is a regular solution of the Dirichlet problem
(\ref{eqDIRICHLET}) for the Beltrami equation (\ref{B}) with
$\mu=\mu_B$.
\end{theorem}

\begin{corollary}\label{cor7} Let $D$ be a bounded simply connected domain in $\mathbb
R^2$, let $B:D\to\mathbb B^{2\times 2}$ be a measurable function in
$D$ and let $\mu=\mu_B$ be given by the formula (\ref{Mu}) with
$K_{\mu_B}\in L^1(D)$. Suppose that
\begin{equation}\label{eqMEANmA}\overline{\lim\limits_{\varepsilon\to0}}\quad
\dashint_{B(z_0,\varepsilon)}K^T_{\mu_B}(z,z_0)\,d{\cal
L}(z)<\infty\qquad\forall\ z_0\in\overline{D}\, .\end{equation} Then
all conclusions of Theorem \ref{th5} hold.
\end{corollary}

\begin{corollary}\label{cor8} Let $D$ be a bounded simply connected domain in $\mathbb
R^2$, let $B:D\to\mathbb B^{2\times 2}$ be a measurable function in
$D$ and let $\mu=\mu_B$ be given by the formula (\ref{Mu}). Suppose
that $K_{\mu_B}$ has a dominant $Q:\mathbb C\to[1,\infty)$ in the
class {\rm BMO}$(\overline D)$. Then all conclusions of Theorem
\ref{th5} hold.
\end{corollary}

\begin{remark}\label{rem9}
In particular, all these conclusions hold if $Q\in{\rm
W}^{1,2}(\overline D)$ because $W^{\,1,2}_{\rm loc} \subset {\rm
VMO}(\overline D)$.
\end{remark}

\begin{corollary}\label{cor9} Let $D$ be a bounded simply connected domain in $\mathbb
R^2$, let $B:D\to\mathbb B^{2\times 2}$ be a measurable function in
$D$ and let $\mu=\mu_B$ be given by the formula (\ref{Mu}). Suppose
that $K_{\mu_B}(z)\leqslant Q(z)$ a.e. in $D$ with a function $Q$ in
the class ${\rm FMO}(\overline{D})$. Then all the given conclusions
hold.
\end{corollary}

\begin{theorem}\label{th6} Let $D$ be a bounded simply connected domain in $\mathbb
R^2$, let $B:D\to\mathbb B^{2\times 2}$ be a measurable function in
$D$ and let $\mu=\mu_B$ be given by the formula (\ref{Mu}) with
$K_{\mu_B}\in L^1(D)$. Suppose that
\begin{equation}\label{eqLOGmA}
\int\limits_{\varepsilon<|z-z_0|<\varepsilon_0}K^T_{\mu_B}(z,z_0)\,\frac{d{\cal
L}(z)}{|z-z_0|^2}
=o\left(\left[\log\frac{1}{\varepsilon}\right]^2\right)\qquad\qquad\forall\
z_0\in\overline{D}\end{equation} as $\varepsilon\to 0$ for some
$\varepsilon_0=\varepsilon(z_0)>0$.

Then the generalized Cauchy-Riemann equation (\ref{BM}) has a
regular solution $(u,v)$ of the Dirichlet problem (\ref{eqDIR}) in
$D$ for each continuous inconstant boundary data $\varphi:\partial
D\to{\Bbb R}$ such that $u={\rm Re}\,f$ and $v={\rm Im}\,f$, where
$f$ is a regular solution of the Dirichlet problem
(\ref{eqDIRICHLET}) for the Beltrami equation (\ref{B}) with
$\mu=\mu_B$.
\end{theorem}

\begin{remark}\label{rem10} Choosing in Lemma~\ref{lem3} the function
$\psi(t)=1/(t\log{1/t})$ instead of $\psi(t)=1/t$, we are able to
replace (\ref{eqLOGmA}) by
\begin{equation}\label{eqLOGLOGmA}
\int\limits_{\varepsilon<|z-z_0|<\varepsilon_0}\frac{K^T_{\mu_B}(z,z_0)\,d{\cal
L}(z)} {\left(|z-z_0|\log{\frac{1}{|z-z_0|}}\right)^2}
=o\left(\left[\log\log\frac{1}{\varepsilon}\right]^2\right)\end{equation}
In general, we are able to give here the whole scale of the
corresponding conditions in terms of iterated logarithms, i.e. using
suitable functions $\psi(t)$ of the form
$1/(t\log{1}/{t}\cdot\log\log{1}/{t}\cdot\ldots\cdot\log\ldots\log{1}/{t})$.
\end{remark}

As above, $k_{\mu_B}^T(z_0,r)$ denotes further the integral mean of
$K^T_{{\mu_B}}(z,z_0)$ over the circle $S(z_0,r)\, :=\, \{ z
\in\mathbb C:\, |z-z_0|\,=\, r\}$.

\begin{theorem}\label{th7} Let $D$ be a bounded simply connected domain in $\mathbb
R^2$, $B:D\to\mathbb B^{2\times 2}$ be a measurable function in $D$
and $\mu=\mu_B$ be given by the formula (\ref{Mu}) with
$K_{\mu_B}\in L^1(D)$. Suppose that, for some
$\varepsilon_0=\varepsilon(z_0)>0$,
\begin{equation}\label{eqLEHTOmA}\int\limits_{0}^{\varepsilon_0}
\frac{dr}{rk^T_{\mu_B}(z_0,r)}=\infty\qquad\forall\
z_0\in\overline{D}\ .\end{equation}

Then the generalized Cauchy-Riemann equation (\ref{BM}) has a
regular solution $(u,v)$ of the Dirichlet problem (\ref{eqDIR}) in
$D$ for each continuous inconstant boundary data $\varphi:\partial
D\to{\Bbb R}$ such that $u={\rm Re}\,f$ and $v={\rm Im}\,f$, where
$f$ is a regular solution of the Dirichlet problem
(\ref{eqDIRICHLET}) for the Beltrami equation (\ref{B}) with
$\mu=\mu_B$.
\end{theorem}

\begin{corollary}\label{cor10} Let $D$ be a bounded simply connected domain in $\mathbb
R^2$, $B:D\to\mathbb B^{2\times 2}$ be a measurable function in $D$
and $\mu=\mu_B$ be given by the formula (\ref{Mu}) with
$K_{\mu_B}\in L^1(D)$. Suppose that
\begin{equation}\label{eqLOGkMA}k^T_{\mu_B}(z_0,\varepsilon)=O\left(\log\frac{1}{\varepsilon}\right)
\qquad\mbox{as}\ \varepsilon\to0\qquad\forall\ z_0\in\overline{D}\
.\end{equation} Then all conclusions of Theorem \ref{th7} hold.
\end{corollary}

\begin{remark}\label{rem11} In particular, the conclusion of Corollary~\ref{cor10} holds
if
\begin{equation}\label{eqLOGKmA} K^T_{\mu_B}(z,z_0)=O\left(\log\frac{1}{|z-z_0|}\right)\qquad{\rm
as}\quad z\to z_0\quad\forall\ z_0\in\overline{D}\,.\end{equation}
Moreover, the condition (\ref{eqLOGkMA}) can be replaced by the
whole series of more weak conditions
\begin{equation}\label{edLOGLOGkMA}
k^T_{\mu_B}(z_0,\varepsilon)=O\left(\left[\log\frac{1}{\varepsilon}\cdot\log\log\frac{1}
{\varepsilon}\cdot\ldots\cdot\log\ldots\log\frac{1}{\varepsilon}
\right]\right) \qquad\forall\ z_0\in \overline{D}\ .
\end{equation}
\end{remark}

\begin{theorem}\label{th8} Let $D$ be a bounded simply connected domain in $\mathbb
R^2$, $B:D\to\mathbb B^{2\times 2}$ be a measurable function in $D$
and $\mu=\mu_B$ be given by the formula (\ref{Mu}). Suppose that
\begin{equation}\label{eqINTEGRALmA}\int\limits_{U_{z_0}}\Phi_{z_0}\left(K^T_{\mu_B}(z,z_0)\right)\,d{\cal L}(z)<\infty
\qquad\forall\ z_0\in \overline{D}\end{equation} for a neighborhood
$U_{z_0}$ of $z_0$ and a convex non-decreasing function
$\Phi_{z_0}:[0,\infty]\to[0,\infty]$ such that,  for some
$\Delta(z_0)>0$,
\begin{equation}\label{eqINTmA}
\int\limits_{\Delta(z_0)}^{\infty}\log\,\Phi_{z_0}(t)\,\frac{dt}{t^2}\
=\ +\infty\ .\end{equation}

Then the generalized Cauchy-Riemann equation (\ref{BM}) has a
regular solution $(u,v)$ of the Dirichlet problem (\ref{eqDIR}) in
$D$ for each continuous inconstant boundary data $\varphi:\partial
D\to{\Bbb R}$ such that $u={\rm Re}\,f$ and $v={\rm Im}\,f$, where
$f$ is a regular solution of the Dirichlet problem
(\ref{eqDIRICHLET}) for the Beltrami equation (\ref{B}) with
$\mu=\mu_B$.
\end{theorem}

\begin{corollary}\label{cor11} Let $D$ be a bounded simply connected domain in $\mathbb
R^2$, $B:D\to\mathbb B^{2\times 2}$ be a measurable function in $D$
and $\mu=\mu_B$ be from (\ref{Mu}). Suppose that
\begin{equation}\label{eqEXPmA}\int\limits_{U_{z_0}}{\rm exp}\,\left[{\alpha(z_0) K^T_{\mu_B}(z,z_0)}\right]\ d{\cal L}(z)<\infty
\qquad\forall\ z_0\in \overline{D}\end{equation} for some
$\alpha(z_0)>0$ and a neighborhood $U_{z_0}$ of the point $z_0$.
Then all conclusions of Theorem \ref{th8} hold. \end{corollary}

\begin{corollary}\label{cor12} Let $D$ be a bounded simply connected domain in $\mathbb
R^2$, $B:D\to\mathbb B^{2\times 2}$ be a measurable function in $D$
and $\mu=\mu_B$ be from (\ref{Mu}). Suppose that
\begin{equation}\label{eqINTKmA}\int\limits_{D}\Phi\left(K_{\mu_B}(z)\right)\,d{\cal L}(z)<\infty\end{equation}
for a convex non-decreasing function $\Phi:[0,\infty]\to[0,\infty]$
with
\begin{equation}\label{eqINTFmA}
\int\limits_{\delta}^{\infty}\log\,\Phi(t)\,\frac{dt}{t^2}\ =\
+\infty\end{equation} for some $\delta>0$. Then the conclusions of
Theorem \ref{th8} hold.
\end{corollary}

\begin{remark}\label{rem12}
By the Stoilow theorem, see e.g. \cite{Sto}, a multi-valued solution
$f=u+iv$ of the Dirichlet problem (\ref{eqDIRICHLET}) in $D$ for the
Beltrami equation (\ref{B}) with $K_{\mu_B}\in L^1_{\rm loc}(D)$ can
be represented in the form $f={\cal A}\circ F$ where $\cal A$ is a
multi-valued analytic function and $F$ is a homeomorphic solution of
(\ref{B}) with $\mu :=\mu_B$ in the class $W_{\rm loc}^{1,1}$. Thus,
by Theorem 5.1 in \cite{RSY3}, see also Theo\-rem 16.1.6 in
\cite{AIM}, the condition (\ref{eqINTFmA}) is not only sufficient
but also necessary to have regular solutions $(u,v)$ of the
Dirichlet problem (\ref{eqDIR}) in $D$ to the generalized
Cauchy-Riemann equation (\ref{BM}) with the integral constraints
(\ref{eqINTKmA}) for all continuous inconstant data
$\varphi:\partial D\to\Bbb{R}$. \end{remark}

\begin{corollary}\label{cor13} Let $D$ be a bounded simply connected domain in $\mathbb
R^2$, $B:D\to\mathbb B^{2\times 2}$ be a measurable function in $D$
and $\mu=\mu_B$ be from (\ref{Mu}). Suppose that, for some
$\alpha>0$,
\begin{equation}\label{eqEXPAmA}\int\limits_{D}{\rm exp}\, {\alpha K_{\mu_B}(Z)}\ d{\cal L}(Z)\ <\
\infty\ .
\end{equation}

Then the generalized Cauchy-Riemann equation (\ref{BM}) has a
regular solution $(u,v)$ of the Dirichlet problem (\ref{eqDIR}) in
$D$ for each continuous inconstant boundary data $\varphi:\partial
D\to{\Bbb R}$ such that $u={\rm Re}\,f$ and $v={\rm Im}\,f$, where
$f$ is a regular solution of the Dirichlet problem
(\ref{eqDIRICHLET}) for the Beltrami equation (\ref{B}) with
$\mu=\mu_B$.
\end{corollary}

Finally note that Remark \ref{remH} remains valid for all these
consequences of Lemma \ref{lem1}.

\section{Boundary value problems with measurable data}

In the case of the uniform ellipticity, i.e., in the case of the
matrix va\-lu\-ed measurable coefficients $B$ in (\ref{BM}) with the
complex criterion $\|\mu\|_{\infty}<1$ in (\ref{Mu}), we are able to
obtain the corresponding results to the generalized Cauchy-Riemann
equations with arbitrary measurable boundary data but not only for
continuous boundary data, and not only for the Dirichlet problem but
for boundary value problems of Hilbert, Neumann, Poincare and
Riemann, also.

\medskip

In in this section, we apply as a basis article \cite{GRY+}, where
it was given a certain pre-history of the given problems and proved
some theorems on existence, representation and regularity of
solutions of these boundary value problems whose data were
measurable with respect to logarithmic capacity for the Beltrami
equations (\ref{B}) with $\|\mu\|_{\infty}<1$ in arbitrary Jordan
domains $D$ in $\mathbb C$.

\medskip

Let $D$ be a domain in $\mathbb R^2$ whose boundary consists of a
finite collection of mutually disjoint Jordan curves. A family of
mutually disjoint Jordan arcs $J_{\zeta}:[0,1]\to\overline D$,
$\zeta\in\partial D$, with $J_{\zeta}([0,1))\subset D$ and
$J_{\zeta}(1)=\zeta$, that is continuous in the parameter $\zeta$,
is called a {\bf Bagemihl--Seidel system}, abbr., {\bf  of class}
${\cal{\bf BS}}$, cf. \cite{BS}, p. 740-741.

\medskip

Arguing similarly to the proof of Lemma~\ref{lem1}, again on the
basis of Proposition \ref{pr} and Remark \ref{rem}, we obtain the
following consequences of the corresponding theorems in \cite{GRY+}.

\medskip

First of all, let us start from the consequence  of Theorem 2 in
\cite{GRY+} for the Dirichlet problem to the Cauchy-Riemann
equations.

\begin{theorem}\label{th5.1}  Let $D$ be a bounded domain in $\mathbb
R^2$ whose boundary consists of a finite number of mutually disjoint
Jordan curves, $B:D\to\mathbb B^{2\times 2}$ be a measurable
function in $D$ with $\|\mu\|_{\infty} < 1$, $\mu=\mu_B$ in
(\ref{Mu}), and let a function $\Phi : \partial D \to \mathbb R^2$
be measurable with respect to logarithmic capacity.

Suppose that $\{\gamma_{\zeta}\}_{\zeta \in \partial D}$ is a family
of Jordan arcs of class $\mathcal{BS}$ in $D$. Then the
Cauchy-Riemann equation (\ref{BM}) has a regular solution
$W(Z)=(u(Z),v(Z))$, $Z=(x,y)$, with the Dirichlet boundary condition
\begin{equation}\label{eqDIRC}
\lim\limits_{Z\to\zeta}W(Z)\ =\ \Phi(\zeta)
\end{equation}
along $\gamma_{\zeta}$ for a.e. $\zeta \in
\partial D$ with respect to logarithmic capacity.
\end{theorem}

Recall that the {\bf classical boundary value problem of Hilbert},
was formulated as follows: To find an analytic function $f(z)$ in a
Jordan domain $D$ in $\mathbb C$ with a rectifiable boundary that
satisfies the condition
\begin{equation}\label{Hilbert}
\lim\limits_{z\to\zeta , z\in D} \ {\rm{Re}}\,
\{\overline{\lambda(\zeta)}\ f(z)\}\ =\ \varphi(\zeta)
\quad\quad\quad
\end{equation}
for all $\zeta\in\partial D$, where the {\bf coefficient} $\lambda$
and the {\bf boundary date} $\varphi$ of the problem were assumed
con\-ti\-nu\-ous\-ly differentiable with respect to the natural
parameter $s$ and $\lambda\ne 0$ everywhere on $\partial D$. The
latter allows to consider that $|\lambda|\equiv 1$ on $\partial D$.
Note that the quantity ${\rm{Re}}\,\{\overline{\lambda}\, f\}$ in
(\ref{Hilbert}) means a projection of $f$ into the direction
$\lambda$ interpreted as vectors in $\mathbb R^2$.

\medskip

The following consequence  of Theorem 4 and Remark 2 in \cite{GRY+}
solves the Hilbert type problem to the Cauchy-Riemann equations
(\ref{BM}).

\begin{theorem}\label{th5.2} Let $D$ be a bounded domain in $\mathbb
R^2$ whose boundary consists of a finite number of mutually disjoint
Jordan curves, $B:D\to\mathbb B^{2\times 2}$ be a measurable
function in $D$ with $\|\mu\|_{\infty} < 1$, $\mu=\mu_B$ in
(\ref{Mu}), and let functions $\varphi : \partial D \to \mathbb R$
and $\Lambda : \partial D \to \mathbb R^2$, $|\Lambda |\equiv 1$, be
measurable with respect to logarithmic capacity.

Suppose that $\{\gamma_{\zeta}\}_{\zeta \in \partial D}$ is a family
of Jordan arcs of class $\mathcal{BS}$ in $D$. Then the
Cauchy-Riemann equation (\ref{BM}) has a regular solution
$W(Z)=(u(Z),v(Z))$, $Z=(x,y)$, with the Hilbert boundary condition
\begin{equation}\label{eqHC}
\lim\limits_{Z\to\zeta}\,\langle\, \Lambda(Z)\,, \,W(Z)\,\rangle\ =\
\varphi(\zeta)
\end{equation}
along $\gamma_{\zeta}$ for a.e. $\zeta \in
\partial D$ with respect to logarithmic capacity.

Furthermore, the space of all such solutions has the infinite
dimension for each such prescribed $B$, $\Lambda$, $\varphi$ and
$\{\gamma_{\zeta}\}_{\zeta \in \partial D}$.
\end{theorem}

Again, on the basis of Theorem 7 in \cite{GRY+}, we also obtain the
fol\-lo\-wing consequence to the genegalized Cauchy-Riemann
equations (\ref{BM}) on the known {\bf Poincare boundary value
problem on directional derivatives}
\begin{equation}\label{4.5}
\frac{\partial W}{\partial {\cal N}}\ :=\ \lim_{t\to 0}\
\frac{W(Z+t\cdot\,{\cal N})-W(Z)}{t}
\end{equation} along the corresponding directing vectors ${\cal N} = {\cal N}(\zeta)$, $\zeta\in\partial
D$.

\begin{theorem}\label{th5.3} Let $D$ be a Jordan domain in $\mathbb
R^2$, $B:D\to\mathbb B^{2\times 2}$ be a matrix function of the
Holder class $C^{\alpha}$, $\alpha \in (0,1)$, with
$\|\mu\|_{\infty} < 1$, $\mu=\mu_B$ in (\ref{Mu}), and let functions
$\Phi : \partial D \to \mathbb R^2$ and ${\cal N} :
\partial D \to \mathbb R^2$, $|{\cal N} |\equiv 1$, be measurable
with respect to logarithmic capacity.

Suppose that $\{\gamma_{\zeta}\}_{\zeta \in \partial D}$ is a family
of Jordan arcs of class $\mathcal{BS}$ in $D$. Then the
Cauchy-Riemann equation (\ref{BM}) has a regular smooth solution
$W(Z)=(u(Z),v(Z))$, $Z=(x,y)$, of the class $C^{1+\alpha}$ with the
Poincare boundary condition
\begin{equation}\label{eqPC}
\lim_{Z \to \zeta}\frac{\partial W(Z)}{\partial {\cal N(\zeta)}}\ =\
\Phi(\zeta)
\end{equation}
along $\gamma_{\zeta}$ for a.e. $\zeta \in
\partial D$ with respect to logarithmic capacity.
\end{theorem}

\begin{remark}\label{remN}
We are able to say more in the case $\langle n , {\cal N}\rangle
> 0$ where $n = n(\zeta)$ is the unit interior normal with a tangent
to $\partial D$ at a point $\zeta \in \partial D$. In view of
(\ref{eqPC}), since the limit $\Phi(\zeta)$ is finite, there is a
finite limit $W(\zeta)$ of $W(Z)$ as $Z \to \zeta$ in $D$ along the
straight line passing through the point $\zeta$ and being parallel
to the vector ${\cal N}(\zeta)$ because along this line, for $Z$ and
$Z_0$ that are close enough to $\zeta$,
\begin{equation}
W(Z)\ =\ W(z_0)\ -\ \int_0^1 \frac{\partial W}{\partial {\cal N}}
\left( Z_0\, +\, \tau\cdot (Z - Z_0) \right) d\tau\ .
\end{equation}
Thus, at each point with the condition (\ref{eqPC}), there is the
derivative
\begin{equation}
\frac{\partial W}{\partial {\cal N}} (\zeta)\ :=\ \lim_{t \to 0}\
\frac{W(\zeta\ +\ t\cdot {\cal N}) - W(\zeta)}{t}\ =\ \Phi(\zeta)\ .
\end{equation}
\end{remark}

Thus, the next result, concerning to the Neumann problem for the
generalized Cauchy-Riemann equations (\ref{BM}), follows from
Theorem \ref{th5.3}.


\begin{corollary}\label{corN}
Let $\mathbb D$ be the unit disk in $\mathbb R^2$ centered at the
origin, ${\rm B}:\mathbb D\to\mathbb B^{2\times 2}$ be a matrix
function of the Holder class $C^{\alpha}$, $\alpha \in (0,1)$, with
$\|\mu\|_{\infty} < 1$, $\mu=\mu_B$ in (\ref{Mu}), and let a
function $\Phi : \partial \mathbb D \to \mathbb R^2$ be measurable
with respect to logarithmic capacity.

\medskip

Then the genegalized Cauchy-Riemann equation (\ref{BM}) has a
regular smooth solution $W(Z)=(u(Z),v(Z))$, $Z=(x,y)$, of the class
$C^{1+\alpha}$ in $\mathbb D$ such that a.e. on $\partial \mathbb D$
there exist:

\medskip

1) the finite limit along the inner normal $n=n(\zeta)$ to $\partial
\mathbb D$ at $\zeta$
\begin{equation}\label{4.10}
W(\zeta)\ :=\ \lim_{Z\to\zeta}\ W(Z)
\end{equation}

2) the normal derivative
\begin{equation}\label{4.11}
\frac{\partial W}{\partial n}\, (\zeta)\ :=\ \lim_{t\to0}\
\frac{W(\zeta+t\cdot n)\ -\ W(\zeta)}{t}\ =\ \Phi(\zeta)
\end{equation}

3) the radial limit \begin{equation}\label{4.12}
 \lim_{z\to\zeta}\ \frac{\partial W}{\partial
n}\, (Z)\ =\ \frac{\partial W}{\partial n}\, (\zeta)\
.\end{equation}
\end{corollary}

The classical setting of the {\bf Riemann boundary value problem} in
a smooth Jordan domain $D$ of the complex plane $\mathbb{C}$ is to
find analytic functions ${\cal A}^+: D\to\mathbb C$ and ${\cal
A}^-:\mathbb C\setminus \overline{D}\to\mathbb C$ that admit
continuous extensions to $\partial D$ and satisfy the boundary
condition
\begin{equation}\label{eqRIEMANN} {\cal A}^+(\zeta)\ =\
a(\zeta)\cdot {\cal A}^-(\zeta)\ +\ b(\zeta)\end{equation} for all
$\zeta\in\partial D$ with its prescribed H\"older continuous
coefficients $a:
\partial D\to\mathbb C$ and $b: \partial D\to\mathbb C$.
Recall also that the {\bf Riemann boundary value problem with shift}
in $D$ is to find analytic functions ${\cal A}^+: D\to\mathbb C$ and
${\cal A}^-:\mathbb C\setminus \overline{D}\to\mathbb C$ such that
\begin{equation}\label{eqSHIFT} {\cal A}^+(\alpha(\zeta))\ =\ a(\zeta)\cdot
{\cal A}^-(\zeta)\ +\ b(\zeta)
\end{equation}
for all $\zeta\in\partial D$, where $\alpha :\partial D\to\partial
D$ was a one-to-one sense preserving correspondence having the
non-vanishing H\"older continuous derivative with respect to the
natural parameter on $\partial D$. The function $\alpha$ is called a
{\bf shift function}. The special case $a\equiv 1$ gives the
so--called {\bf jump problem} and then $b\equiv 0$ the {\bf problem
on gluing} of analytic functions.

\medskip

Now, on the basis of Theorem 9 in \cite{GRY+}, we obtain the
fol\-lo\-wing consequence on the Riemann boundary value problem to
the generalized Cauchy-Riemann equations (\ref{BM}).

\begin{theorem}\label{th5.4} Let $D$ be a domain in $\mathbb
R^2$ whose boundary consists of a finite number of mutually disjoint
Jordan curves, $B:\mathbb R^2\to\mathbb B^{2\times 2}$ be a
measurable function in $\mathbb R^2$ with $\|\mu\|_{\infty} < 1$,
$\mu=\mu_B$ is from (\ref{Mu}), and let functions $a :
\partial D \to \mathbb M^{2\times 2}$ and $b :
\partial D \to \mathbb R^2$ be measurable with respect to
logarithmic capacity. Suppose also that $\{\gamma_{\zeta}^+\}_{\zeta
\in \partial D}$ and $\{\gamma_{\zeta}^-\}_{\zeta \in \partial D}$
are families of Jordan arcs of class $\mathcal{BS}$ in $D$ and
$\mathbb{R}^2 \setminus \overline{D}$, correspondingly.

\medskip

Then the genegalized Cauchy-Riemann equation (\ref{BM}) has its
regular solutions
\newline $W^{\pm}(Z)=(u^{\pm}(Z),v^{\pm}(Z))$, $Z=(x,y)$,  in $D$ and
$\mathbb{R}^2 \setminus \overline{D}$, correspondingly, with the
Riemann boundary condition
\begin{equation}\label{eqR} W^+(\zeta)\ =\
a(\zeta)\cdot W^-(\zeta)\ +\ b(\zeta)\end{equation} for a.e. $\zeta
\in \partial D$ with respect to logarithmic capacity, where
$W^{\pm}(\zeta)$ are limits of $W^{\pm}(Z)$ as $Z \to \zeta$ along
$\gamma_{\zeta}^{\pm}$, correspondingly, and vectors in $\mathbb
R^2$ are interpreted as vector-columns.

\medskip

Furthermore, the space of all such solutions has the infinite
dimension for each such prescribed coefficients $B$, $a$, $b$ and
collections $\{\gamma_{\zeta}^{\pm}\}_{\zeta \in \partial D}$.
\end{theorem}

\begin{remark}\label{remS}
The similar statement to the generalized Cauchy-Riemann equations
(\ref{BM}) can be formulated on the basis of Lemma 4 in \cite{GRY+}
for the Riemann boundary problem with homeomorphic shifts $\alpha
:\partial D\to\partial D$ keeping components of $\partial D$ such
that $\alpha$ and $\alpha^{-1}$ have the (N)-property by Lusin with
respect to logarithmic capacity.

\medskip

Moreover, some investigations on analytic functions were devoted
also to the nonlinear Riemann problems. As a consequence of Remark 8
in \cite{GRY+}, such a Riemann's problem with a nonlinear boundary
condition
\begin{equation}\label{eqNON}
W^+(\zeta)\ =\ \Phi(\,\zeta\, ,\, W^-(\zeta)\,)\ \ \ \ \ \mbox{for
a.e. $\zeta \in \partial D$}
\end{equation}
is solved to the generalized Cauchy-Riemann equations (\ref{BM})
under the hypotheses of Theorem \ref{th5.4} if the function
$\Phi:\partial D\times\mathbb R^2\to\mathbb R^2$ satisfies the {\bf
Ca\-ra\-theo\-do\-ry conditions} with respect to logarithmic
capacity, i.e., if $\Phi(\zeta, w)$ is continuous in the variable
$w\in\mathbb R^2$ for a.e. $\zeta\in\partial D$ with respect to the
logarithmic capacity and it is measurable with respect to
logarithmic capacity in the variable $\zeta\in\partial D$ for each
$w\in\mathbb R^2$. Furthermore, the space of all regular solutions
of the problem has the infinite dimension.
\end{remark}

Finally, in order to demonstrate the potentiality of our approach,
we give here also one more consequence for other nonlinear boundary
value problem of a mixed Poincare-Riemann type to the generalized
Cauchy-Riemann equations (\ref{BM}) from Theorem 10 in \cite{GRY+}.

\begin{theorem}\label{th5.5} Let $D$ be a domain in $\mathbb
R^2$ whose boundary consists of a finite number of mutually disjoint
Jordan curves, $B:\mathbb R^2\to\mathbb B^{2\times 2}$ be a
measurable function in $\mathbb R^2$ with $\|\mu\|_{\infty} < 1$,
$\mu \in C^{\alpha}(\mathbb{R}^2 \setminus \overline{D})$,
$\mu=\mu_B$ in (\ref{Mu}), let $\Phi:\partial D\times\mathbb
R^2\to\mathbb R^2$ satisfy the Ca\-ra\-theo\-do\-ry conditions and a
function ${\cal N} : \partial D \to \mathbb{R}^2$, $|{\cal
N}(\zeta)| \equiv 1$, be measurable with respect to logarithmic
capacity. Suppose also that $\{\gamma_{\zeta}^+\}_{\zeta \in\partial
D}$ and $\{\gamma_{\zeta}^-\}_{\zeta \in \partial D}$ are families
of Jordan arcs of class $\mathcal{BS}$ in $D$ and $\mathbb{R}^2
\setminus \overline{D}$, correspondingly.

\medskip

Then the genegalized Cauchy-Riemann equation (\ref{BM}) has its
regular solutions
\newline $W^{\pm}(Z)=(u^{\pm}(Z),v^{\pm}(Z))$, $Z=(x,y)$,  in $D$ and
$\mathbb{R}^2 \setminus \overline{D}$, correspondingly, $W^-\in
C^{1+\alpha}(\mathbb{R}^2 \setminus \overline{D})$, with the
Poincare-Riemann boundary condition
\begin{equation}\label{eqPR}
W^+(\zeta)\ =\ \Phi\left(\zeta, \left[\frac{\partial W}{\partial
{\cal N}}\right]^-(\zeta)\right)
\end{equation}
for a.e. $\zeta \in \partial D$ with respect to logarithmic
capacity, where $W^+(\zeta)$ and $\left[\frac{\partial W}{\partial
{\cal N}}\right]^-(\zeta)$ denote limits of the functions $W^+(Z)$
and $\frac{\partial W^-}{\partial {\cal N}}(Z)$ as $Z \to \zeta$
along $\gamma_{\zeta}^+$ and $\gamma_{\zeta}^-$, correspondingly.

\medskip

Furthermore, the space of all such solutions has the infinite
dimension for each prescribed collection $B$, $\Phi$, ${\cal N}$ and
$\{\gamma_{\zeta}^{\pm}\}_{\zeta \in \partial D}$.
\end{theorem}

\begin{remark}\label{remPR}
As in Remark \ref{remS}, the similar consequence to the
ge\-ne\-ra\-li\-zed Cauchy-Riemann equations (\ref{BM}) can be
formulated on the basis of Lemma 5 in \cite{GRY+} on the nonlinear
Poincare-Riemann boundary problem with homeomorphic shifts $\alpha
:\partial D\to\partial D$ keeping components of $\partial D$ such
that $\alpha$ and $\alpha^{-1}$ have the (N)-property by Lusin with
respect to logarithmic capacity.
\end{remark}

\section{Boundary value problems and angular limits}

In the last section, you found theorems on boundary value problems
to generalized Cauchy-Riemann equation (\ref{BM}), where the
boundary values are meant as limits along the systems of Jordan arcs
of the Bagemihl-Seidel class that contained exactly one such an arc
terminating at each boun\-da\-ry point. At the present section, we
obtain the corresponding consequences on the boundary value problems
of Dirichlet and Hilbert to generalized Cauchy-Riemann equations
(\ref{BM}), where boundary values are understood as li\-mits along
all non-tangent paths to each boundary point. The latter request is
formally more strong but except tangent arcs to the boundary points.
Thus, these results cannot be derived each from other. Note also
that in the both cases conclusions on boundary points hold only a.e.
with respect of logarithmic capacity.

\medskip

For further definitions and historic comments, we refer the reader
to the article \cite{GRYY+}, which is a basis for us to obtain a new
series of consequences on boundary value problems for the
Cauchy-Riemann equations (\ref{BM}) from the theory of the Beltrami
equations (\ref{B}). Note only here that each smooth (or Lipschitz)
domain satisfies the quasihyperbolic boundary condition by
Gehring-Martio from the article \cite{GM} but such boundaries can be
even nowhere locally rectifiable quasicircles.

\medskip

Moreover, domains with the $(A)-$condition by
Ladyzhenskaya–Ural’\-tse\-va, in particular, with the well-known
outer cone condition, which is standard in the theory of boundary
value problems for PDE, see e.g. \cite{LU}, satisfy the mentioned
quasihyperbolic boundary condition. Probably one of the simplest
examples of a domain $D$ with the quasihyperbolic boundary condition
and without (A)--condition is the union of 3 open disks with the
radius 1 centered at the points $0$ and $1\pm i$. It is clear that
the domain has zero interior angle at its boundary point $1$.

\medskip

Similarly to the last section, on the basis of Theorem 6.1, its
proof and Corollary 8.2 in \cite{GRYY+}, and based on Proposition
\ref{pr} and Remark \ref{rem}, we obtain the following consequence
on the Hilbert problem.

\begin{theorem}\label{th6.1} Let $D$ be  a Jordan domain with the quasihyperbolic
boundary condition and $\partial D$ have a tangent a.e. with respect
to logarithmic capacity and  $B:D\to\mathbb B^{2\times 2}$ be
measurable function in $D$ with $\|\mu\|_{\infty} < 1$, $\mu=\mu_B$
in (\ref{Mu}). Suppose that a function $\Lambda :
\partial D \to \mathbb R^2$, $|\Lambda |\equiv 1$, is of
coun\-tab\-ly bounded variation and a function $\varphi :\partial D
\to \mathbb R$ is measurable with respect to logarithmic capacity.

\medskip

Then the genegalized Cauchy-Riemann equation (\ref{BM}) has a
regular solution $W(Z)=(u(Z),v(Z))$, $Z=(x,y)$, in $W^{1,2}_{\rm
loc}(D)$ with the Hilbert condition
\begin{equation}\label{eqHCA}
\lim\limits_{Z\to\zeta}\,\langle\, \Lambda(Z)\,, \,W(Z)\,\rangle\ =\
\varphi(\zeta) \ \ \ \ \ \hbox{for a.e. $\zeta \in
\partial D$ }
\end{equation}
with respect to logarithmic capacity along all nontangent paths.

\medskip

Furthermore, the space of all such solutions has the infinite
dimension for each such prescribed $B$, $\Lambda$ and $\varphi$.
\end{theorem}

In particular case with $\Lambda = (1,0)$ for all $\zeta\in\partial
D$, we obtain from here the corresponding consequence on the
Dirichlet problem.

\begin{corollary}\label{cor6.1} Under all hypotheses of Theorem
\ref{th6.1} with $\Lambda\equiv(1,0)$,
\begin{equation}\label{eqDA}
\lim\limits_{Z\to\zeta}\, u(Z)\ =\ \varphi(\zeta) \ \ \ \ \
\hbox{for a.e. $\zeta \in
\partial D$ }
\end{equation}
with respect to logarithmic capacity along all nontangent paths.
Furthermore, the space of all such solutions has the infinite
dimension for each such prescribed $B$ and $\varphi$.
\end{corollary}

\bigskip

The corresponding survey of consequences on boundary value problems
to generalized Cauchy-Riemann equations with sources from the theory
of the Beltrami equations will be published elsewhere.

\medskip

{\bf Acknowledgments.} The first 2 authors are partially supported
by the project "Mathematical modeling of complex dynamical systems
and processes caused by the state security", No. 0123U100853, of the
National Academy of Scien\-ces of Ukraine and by a grant from the
Simons Foundation PD-Ukraine-00010584.

\bigskip

{\bf \noindent Vladimir Gutlyanskii, Vladimir Ryazanov, Artem Salimov} \\
Institute of Applied Mathematics and Mechanics of the\\
National Academy of Sciences of Ukraine, Slavyansk,\\
vgutlyanskii@gmail.com, vl.ryazanov1@gmail.com, \\
Ryazanov@nas.gov.ua, artem.salimov1@gmail.com

\bigskip

{\bf \noindent Ruslan Salimov}\\
{Institute of Mathematics of the NAS of Ukraine, Kyiv,
Ukraine}\\
ruslan.salimov1@gmail.com


\begin{thebibliography}{100}

\bibitem{Ah}
{\sc  Ahlfors L.} (1966) \emph{Lectures on Quasiconformal Mappings.}
New York: Van Nostrand.

\bibitem{ABe}
{\sc  Ahlfors, L.V., Bers, L.:} (1960) Riemann’s mapping theorem for
variable metrics. {\it Ann. Math. (2) 72}, 385-404.

\bibitem{AC$_2$}
{\sc  Andreian Cazacu C.:} Some formulae on the extremal length in
$n$-dimensional case.  Proc. Rom.-Finn. Sem. on Teichm\"uller Spaces
and Quasiconformal Mappings (Brazov, 1969), pp. 87--102, Publ. House
of Acad. Soc. Rep. Romania, Bucharest (1971).

\bibitem{AIKM}
{\sc  Astala K., Iwaniec T., Koskela P., Martin G.:} Mappings of
BMO-bounded distortion. Math. Ann. 317, No. 4, 703-726 (2000).

\bibitem{AIM}
{\sc  Astala K., Iwaniec T., Martin G.J.:} (2009) \emph{Elliptic
differential equations and quasiconformal mappings in the plane.}
Princeton Math. Ser.  {\bf 48}. Princeton: Princeton Univ. Press.

\bibitem{BS}
{\sc Bagemihl F., Seidel W.:} Regular functions with prescribed
measurable boundary values almost everywhere. - Proc. Nat. Acad.
Sci. U.S.A. 41, 1955, 740–743.


\bibitem{BGMR} {\sc Bojarski B., Gutlyanskii V., Martio O., Ryazanov V.:}
In\-fi\-ni\-te\-si\-mal geometry of quasiconformal and bi-Lipschitz
mappings in the plane. - EMS Tracts in Mathematics, 19.  Z\"urich:
European Mathematical Society, 2013.

\bibitem{BN}
{\sc Brezis H., Nirenberg L.:} {Degree theory and BMO. I. Compact
manifolds without boundaries.} Selecta Math. (N.S.) 1:2, 1995,
197--263.

\bibitem{CFL}
{\sc Chiarenza F., Frasca M., Longo P.:} {$W^{2,p}$-solvability of
the Dirichlet problem for nondivergence elliptic equations with VMO
coefficients.} - Trans. Amer. Math. Soc. 336:2, 1993, 841--853.



\bibitem{GM}
{\sc  Gehring F. W., Martio O.:} Lipschitz classes and
quasiconformal mappings // Ann. Acad. Sci. Fenn. Ser. A I Math. --
1985. -- {\bf 10}. -- P. 203–219.

\bibitem{GMR}
{\sc  Gutlyanskii, V.,  Martio, O.,  Ryazanov, V.:} (2023)
A-harmonic equation and cavitation. Annales Fennici Mathematici 48,
no.1, 277-297.

\bibitem{GMSV}
{\sc Gutlyanskii V., Martio O., Sugawa T., Vuorinen M.:} On the
degenerate Beltrami equation. - Trans. Amer. Math. Soc. 357:3, 2005,
875-900.

\bibitem{GRSY22} {\sc Gutlyanskii V., Ryazanov V., Sevost'yanov E., Yakubov
E.} (2022) BMO and Dirichlet problem for degenerate Beltrami
equation. J. Math. Sci. (New York) 268, No. 2, 157-177; (2022)
translation from Ukr. Mat. Visn. 19, No. 3, 327-354.

\bibitem{GRSY22+} {\sc Gutlyanskii V., Ryazanov V., Sevost'yanov E., Yakubov
E.} (2022) On the degenerate Beltrami equation and hydrodynamic
normalization. J. Math. Sci. (New York) 262, No. 2, 165-183; (2022)
translation from Ukr. Mat. Visn. 19, No. 1, 49-74.

\bibitem{GRSY} {\sc Gutlyanskii V., Ryazanov V., Srebro U., Yakubov
E.:} The Beltrami Equation: A Geometric Approach. - Developments in
Mathematics 26, Springer: Berlin, 2012.

\bibitem{GRY} {\sc Gutlyanskii V., Ryazanov V., Yakubov E.} (2015)
The Beltrami equations and prime ends. J. Math. Sci. (New York) 210,
No. 1, 22-51; (2015) translation from Ukr. Mat. Visn. 12, No. 1,
27-66.

\bibitem{GRY+} {\sc Gutlyanskii V., Ryazanov V., Yefimushkin A.} (2016)
On the boundary-value problems for quasiconformal functions in the
plane. J. Math. Sci. (New York) 214, No. 2, 200-219; (2015)
translation from Ukr. Mat. Visn. 12, No. 3, 363-389.

\bibitem{GRYY} {\sc Gutlyanskii V., Ryazanov V., Yakubov E., Yefimushkin A.}
(2021) On the Hilbert boundary-value problem for Beltrami equations
with singularities. J. Math. Sci. (New York) 254, no. 3, 357-374;
(2020) translation from Ukr. Mat. Visn. 17, no. 4, 484-508.

\bibitem{GRYY+} {\sc Gutlyanskii V., Ryazanov V., Yakubov E., Yefimushkin A.}
(2020) On Hilbert boundary value problem for Beltrami equation. Ann.
Acad. Sci. Fenn., Math. 45, No. 2, 957-973.


\bibitem{ER+}
{\sc  Efimushkin A.S., Ryazanov V.I.:} (2016) On the Riemann-Hilbert
problem for the Beltrami equations. Contemporary Mathematics 667.
Israel Mathematical Conference Proceedings, 299-316.

\bibitem{ER}
{\sc  Efimushkin A.S.; Ryazanov V.I.:} On the Riemann-Hilbert
problem for the Beltrami equations in quasidisks. (2015) J. Math.
Sci. (N. Y.) 211, no. 5, 646-659; (2015) translation from Ukr. Mat.
Visn. 12, no. 2, 190-209.

\bibitem{HKM}
{\sc  Heinonen, J., Kilpel\"ainen, T., Martio, O.:} (1993).
\emph{Nonlinear $A-$harmonic theory of degenerate elliptic
equations}. Oxford Mathematical Monographs. Oxford Science
Publications, The Clarendon Press, Oxford University Press, New
York.

\bibitem{H1}
{\sc   Hilbert D.:} \"Uber eine Anwendung der Integralgleichungen
auf eine Problem der Funktionentheorie. -- Verhandl. des III Int.
Math. Kongr., Heidelberg, 1904.


\bibitem{IR}
{\sc Ignat'ev A.A., Ryazanov V.I.:} {Finite mean oscillation in the
mapping theory.} - Ukrainian Math. Bull. 2:3, 2005, 403-424.

\bibitem{IKM}
{\sc  Iwaniec T., Koskela P., Martin G.:} Mappings of BMO-distortion
and Beltrami-type operators. J. Anal. Math. 88, 337-381 (2002).

\bibitem{IS}
{\sc Iwaniec T., C. Sbordone C.:} {Riesz transforms and elliptic
PDEs with VMO coefficients.} - J. Anal. Math. 74, 1998, 183--212.

\bibitem{JN}
{\sc John F., Nirenberg L.:} {On functions of bounded mean
oscillation.} - Comm. Pure Appl. Math. 14, 1961, 415--426.

\bibitem{KPR1}
{\sc  Kovtonyuk D.A., Petkov I.V., Ryazanov V.I.:} (2013) On the
boun\-da\-ry behaviour of solutions to the Beltrami equations.
Complex Var. Elliptic Equ. 58, no. 5, 647-663.

\bibitem{KPR2}
{\sc  Kovtonyuk D.A., Petkov I.V., Ryazanov V.I.:} (2012) On the
Dirichlet problem for the Beltrami equations in finitely connected
domains. Ukr. Math. J. 64, no. 7, 1064-1077; (2012) translation from
Ukr. Mat. Zh. 64, no. 7, 932-944.

\bibitem{KPR3}
{\sc  Kovtonyuk D.A., Petkov I.V., Ryazanov V.I.} (2012) On the
boundary behavior of solutions of the Beltrami equations. Ukr. Math.
J. 63, no. 8, 1241-1255; (2011) translation from Ukr. Mat. Zh. 63,
no. 8, 1078-1091.

\bibitem{KPRS+}
{\sc  Kovtonyuk D.A., Petkov I.V., Ryazanov V.I., Salimov R.R.:}
(2014) On the Dirichlet problem for the Beltrami equation. J. Anal.
Math. 122, 113-141.

\bibitem{KPRS}
{\sc  Kovtonyuk D.A., Petkov I.V., Ryazanov V.I., Salimov R.R.:}
(2014) Boundary behavior and the Dirichlet problem for Beltrami
equations. St. Petersbg. Math. J. 25, no. 4, 587-603; (2013)
translation from Algebra Anal. 25, No. 4, 101-124.

\bibitem{LU}
{\sc  Ladyzhenskaya O.A., Ural'tseva N.N.} { Linear and
qua\-si\-li\-ne\-ar elliptic equations}. -- Academic Press : New
York - London, 1968.

\bibitem{Le}
{\sc Lehto O.:} Homeomorphisms with a prescribed dilatation. -
Lecture Notes in Math. 118, 1968, 58-73.

\bibitem{LV} {\sc Lehto O., Virtanen K.I.:} Quasiconformal mappings in the plane. -
Die Grundlehren der mathematischen Wissenschaften 126, Springer:
Berlin-Heidelberg-New York, 1973.

\bibitem{MRSY}
{\sc Martio O., Ryazanov V., Srebro U., Yakubov E.:} Moduli in
modern mapping theory. - Springer Monographs in Mathematics. -
Springer: New York, 2009.

\bibitem{MRV}
{\sc  Martio O., Ryazanov V., Vuorinen M.:} BMO and injectivity of
space quasiregular mappings. Math. Nachr. 205, 149-161 (1999).

\bibitem{MU}
{\sc  Menovschikov A., Ukhlov A.:} Composition operators on
Hardy-Sobolev spaces and BMO-quasiconformal mappings. J. Math. Sci.
(New York) 258, No. 3, 313-325 (2021); translation from Ukr. Mat.
Visn. 18, No. 2, 209-225 (2021).

\bibitem{Ne} {\sc Nevanlinna R.:} Eindeutige analytische Funktionen. 2. Aufl. Reprint.
(German) - Die Grundlehren der mathematischen Wissenschaften. Band
46. Springer-Verlag: Berlin-Heidelberg-New York, 1974.

\bibitem{Pal}
{\sc Palagachev D.K.:} {Quasilinear elliptic equations with VMO
coefficients.} - Trans. Amer. Math. Soc. 347:7, 1995, 2481--2493.

\bibitem{Ra$_1$}
{\sc Ragusa M.A.:} {Elliptic boundary value problem in vanishing
mean oscillation hypothesis.} - Comment. Math. Univ. Carolin. 40:4,
1999, 651--663.

\bibitem{Ra$_2$}
{\sc Ragusa M.A., Tachikawa A.:} Partial regularity of the
minimizers of quadratic functionals with VMO coefficients. J. Lond.
Math. Soc., II. Ser. 72:3, 2005, 609-620.

\bibitem{RW}
{\sc  Reich E., Walczak H.:} On the behavior of quasiconformal
mappings at a point. Trans. Amer. Math. Soc. \textbf{117}, 338--351
(1965).

\bibitem{RR}
{\sc Reimann H.M., Rychener T.:} {Funktionen Beschr\"ankter
Mittlerer Oscillation.} - Lecture Notes in Math. 487, 1975.

\bibitem{RS}
{\sc  Ryazanov V.I., Salimov R.R.:} (2007) Weakly planar spaces and
boundaries in the theory of mappings. Ukr. Mat. Visn. 4, no. 2,
199–234, 307; (2007) translation in Ukr. Math. Bull. 4, no. 2,
199–234.

\bibitem{RSY1}
{\sc  Ryazanov V., Srebro U., Yakubov E.:} On boundary value
problems for the Beltrami equations. Contemporary Mathematics 591.
Israel Mathematical Conference Proceedings, 211-242 (2013).

\bibitem{RSY2}
{\sc  Ryazanov V., Srebro U., Yakubov E.:} Integral conditions in
the theory of the Beltrami equations. Complex Var. Elliptic Equ. 57,
No. 12, 1247-1270 (2012).

\bibitem{RSY3}
{\sc  Ryazanov V., Srebro U., Yakubov E.:} On integral conditions in
the mapping theory. J. Math. Sci. (New York), 173, No. 4, 397-407
(2011); translation from Ukr. Mat. Visn. 7, No. 1, 73-87 (2010).

\bibitem{RSY4}
{\sc  Ryazanov V., Srebro U., Yakubov E.:} On strong solutions of
the Beltrami equations. Complex Var. Elliptic Equ. 55, No. 1-3,
219-236 (2010).

\bibitem{RSY5}
{\sc  Ryazanov V., Srebro U., Yakubov E.:} Finite mean oscillation
and the Beltrami equation. Isr. J. Math. 153, 247-266 (2006).

\bibitem{RSY6}
{\sc  Ryazanov V., Srebro U., Yakubov E.:} On the theory of the
Beltrami equation. Ukr. Mat. Zh. 58, No. 11, 1571-1583 (2006);
translation in Ukr. Math. J. 58, No. 11, 1786-1798 (2006).

\bibitem{RSY7}
{\sc  Ryazanov V., Srebro U., Yakubov E.:} On ring solutions of
Beltrami equations. J. Anal. Math. 96, 117-150 (2005).

\bibitem{RSY8}
{\sc  Ryazanov V., Srebro U., Yakubov E.:} The Beltrami equation and
FMO functions. Contemporary Mathematics 382. Israel Mathematical
Con\-fe\-ren\-ce Proceedings, 357-364 (2005).

\bibitem{RSY9}
{\sc  Ryazanov V., Srebro U., Yakubov E.:} Plane mappings with
dilatation dominated by functions of bounded mean oscillation. Sib.
Adv. Math. 11, No. 2, 94-130 (2001).

\bibitem{RSY10}
{\sc  Ryazanov V., Srebro U., Yakubov E.:} BMO-quasiconformal
mappings. J. Anal. Math. 83, 1-20 (2001).


\bibitem{Saks}
{\sc  Saks S.:} Theory of the Integral. Dover, New York (1964).

\bibitem{Sar}
{\sc Sarason D.:} {Functions of vanishing mean oscillation.} -
Trans. Amer. Math. Soc. 207, 1975, 391--405.


\bibitem{Sto}
{\sc Stoilow S.:} {Lecons sur les Principes Topologue de le Theorie
des Fonctions Analytique.} Gauthier-Villars, 1938. Riemann,
Gauthier-Villars, Paris, 1956 [in French].



\end{thebibliography}
\end{document}